\documentclass[a4paper,12pt,oneside]{article} 
\usepackage{amsfonts,amssymb}
\usepackage{color} \usepackage{mathrsfs} \let\mathcal\mathscr
\usepackage[all,ps,cmtip]{xy}

\title{\sc Wonderful resolutions and categorical crepant resolutions of singularities}
\author{\sc Roland Abuaf \footnote{Institut Fourier, 100 rue des maths,
38402, Saint
Martin d'H\`eres, France. E-mail :\it{roland.abuaf@ujf-grenoble.fr}}}

\usepackage[dvips]{graphicx}

\usepackage{pstricks}
\usepackage{amsfonts,amssymb}
\usepackage{color} \usepackage{mathrsfs} \let\mathcal\mathscr

\usepackage[all,ps,cmtip]{xy}

\DeclareGraphicsExtensions{eps}
\DeclareGraphicsExtensions{pdf}

\usepackage[T1]{fontenc}
\usepackage{amssymb}

\usepackage{amsmath}
\usepackage{latexsym}

\usepackage[latin1]{inputenc}

\newtheorem{theo}{Theorem}[subsection]

\newtheorem{claim}[theo]{Claim}

\newtheorem{exem}[theo]{Example}

\newtheorem{prop}[theo]{Proposition}
\newtheorem{quest}[theo]{Question}

\newtheorem{defi}[theo]{Definition}

\newtheorem{lem}[theo]{Lemma}
\newtheorem{cor}[theo]{Corollary}
\newtheorem{conj}[theo]{Conjecture}

\def\DB{\mathrm{D^{b}}}
\def\OO{\mathcal{O}}
\def\w{\omega}

\def\DP{\mathrm{D^{perf}}}
\def\Ri{\mathrm{R^{i}}}
\def\R0{\mathrm{R^{0}}}
\def\ot{\otimes}
\def\Hi{\mathrm{H^i}}

\def\HH{\mathrm{Hom}}
\def\Hh{\mathcal{H}om}
\def\LL{\mathrm{\textbf{L}}}
\def\Li{\mathrm{L^{i}}}
\def\RR{\mathrm{\textbf{R}}}
\def\OO{\mathcal{O}}

\def\d{\delta}
\def\P{\mathcal{P}}
\def\F{\mathcal{F}}
\def\T{\mathcal{T}}
\def\A{\mathcal{A}}

\def\p12{\pi_{{\T_1},{\T_2}}}
\newcommand{\leftexp}[2]{{\vphantom{#2}}^{#1}{#2}}

\newenvironment{proof}
{
\noindent
\textit{\underline{Proof}} :\\
$\blacktriangleright\;$%
}
{\hspace{\stretch{1}}%
$\blacktriangleleft$}

{
\noindent
\textit{\underline{Proof of the main theorem}} :\\
$\blacktriangleright\;$%
}
{\hspace{\stretch{1}}%
$\blacktriangleleft$}

{
\noindent
\textit{\underline{Proof of theorem 3.2.4}} :\\
$\blacktriangleright\;$%
}
{\hspace{\stretch{1}}%
$\blacktriangleleft$}

{
\noindent
\textit{\underline{Proof of prposition A.0.6}} :\\
$\blacktriangleright\;$%
}
{\hspace{\stretch{1}}%
$\blacktriangleleft$}

{
\noindent
\textit{\underline{Road Map}} :\\
$\blacktriangleright\;$%
}
{\hspace{\stretch{1}}%
$\blacktriangleleft$}

{
\noindent
\textit{\underline{Sketch of the Proof}} :\\
$\blacktriangleright\;$%
}
{\hspace{\stretch{1}}%
$\blacktriangleleft$}

{
\noindent
\textit{\underline{Preuve}} : (id\'ee)\\
$\blacktriangleright\;$%
}
{\hspace{\stretch{1}}%
$\blacktriangleleft$}

\begin{document}

\maketitle

\begin{abstract}

Let $X$ be an algebraic variety with Gorenstein singularities. We define the
notion of a \textit{wonderful resolution of singularities} of $X$ by analogy
with the theory of wonderful compactifications of semi-simple linear algebraic
groups. We prove that if $X$ has rational singularities and has a wonderful
resolution of singularities, then $X$ admits a categorical crepant resolution of
singularities. As an immediate corollary, we get that all determinantal
varieties defined by the minors of a generic square/symmetric/skew-symmetric
matrix admit categorical crepant resolution of singularities. 

\end{abstract}

\vspace{\stretch{1}}

\newpage

\tableofcontents

\begin{section}{Introduction}
 
Let $X$ be an algebraic variety over $\mathbb{C}$. Hironaka proved in
\cite{hiro} that one can find a proper birational morphism $ \tilde{X}
\rightarrow X$, with $\tilde{X}$ smooth. Such a $\tilde{X}$ is called a
\textit{resolution of singularities} of $X$. Unfortunately, given an
algebraic variety $X$, there is, in general, no \textit{minimal} resolution of singularities of $X$. In case $X$ is
Gorenstein, a crepant resolution of $X$ (that is a resolution $\pi : \tilde{X}
\rightarrow X$ such that $\pi^* K_X = K_{\tilde{X}}$) is often considered to be
minimal. The conjecture of Bondal-Orlov (see \cite{BO}) gives a precise meaning
to that notion of minimality:

\begin{conj}
 Let $X$ be an algebraic variety with canonical Gorenstein singularities. Assume that $X$ has a crepant resolution of singularities $\tilde{X} \rightarrow X$. Then, for any other resolution of singularities $Y \rightarrow X$, there exists a fully faithful embedding:

\begin{equation*}
\DB (\tilde{X}) \hookrightarrow \DB (Y). 
\end{equation*}
\end{conj}

Varieties admitting a crepant resolution of singularities are quite rare. For instance, non-smooth Gorenstein $\mathbb{Q}$-factorial terminal
singularities (e.g. a cone over $v_2(\mathbb{P}^m) \subset
\mathbb{P}^{\frac{m(m+1)}{2}}$, for even $m$, see \cite{kuz}, section $7$) never
admit crepant resolution of
singularities. Thus, it seems natural to look for minimal resolutions among
\textit{categorical} ones. Kuznetsov has given the following definition (\cite{kuz}):

\begin{defi}
Let $X$ be an algebraic variety with Gorenstein and rational singularities. A
\emph{categorical resolution of singularities} of $X$ is a
triangulated category $\T$ with a functor $\RR {\pi_{\T}}_* : \T \rightarrow
\DB(X)$ such that:
\begin{itemize}
\item there exists a resolution of singularities $\pi : \tilde{X} \rightarrow
X$ such that $\d : \T \hookrightarrow \DB(\tilde{X})$ is admissible and $\RR
{\pi_{\T}}_* = \RR \pi_* \circ \d$,

\item we have $\LL \pi^* \DP(X) \subset \T$ and for all $\F \in \DP(X)$:

\begin{equation*}
\RR {\pi_{\T}}_* \LL \pi_{\T}^* \F \simeq \F,
\end{equation*}
where $\LL \pi_{\T}^*$ is the left adjoint to $\RR {\pi_{\T}}_*$.
\end{itemize}
\bigskip

\noindent If for all $\F \in \DP(X)$, there is a quasi-isomorphism:
\begin{equation*}
\LL \pi_{\T}^* \F \simeq \LL \pi_{\T}^! \F,
\end{equation*}
where $\LL \pi_{\T}^!$ is the right adjoint of $\RR {\pi_{\T}}_*$, we say that
$\T$ is \emph{weakly crepant}.

\bigskip

\noindent Finally, if $\T$ has a structure of module category over $\DP(X)$ and the identity is a relative Serre functor for $\T$ with respect to $\DB(X)$, then $\T$ is said to be \emph{strongly
crepant}.
\end{defi} 

Obviously, if $\T$ is a strongly crepant resolution of $X$, then it is also a weakly crepant resolution of $X$. The converse is false, as shown is section $8$ of \cite{kuz}. If $\pi : \tilde{X} \rightarrow X$ is a crepant resolution of $X$, the one easily shows that $\DB(\tilde{X}) \rightarrow \DB(X)$ is a strongly crepant categorical resolution. The main result of \cite{kuz} is the:

\begin{theo} \label{kuzmaintheo}
Let $X$ be an algebraic variety with Gorenstein rational singularities. Let $\pi : \tilde{X} \rightarrow X$ be a resolution of singularities with a positive integer $m$ such that $K_{\tilde{X}} = \pi^* K_X \ot \OO_{\tilde{X}}(mE)$, where $E$ is the scheme-theoretic exceptional divisor of $\pi$. Assume moreover that we have a semi-orthogonal decomposition:

\begin{equation*}
\DB(E) = \langle \OO_{E}(mE) \ot B_m, \ldots, \OO_{E}(E) \ot B_1, B_0 \rangle,
\end{equation*}
with:
\begin{equation*}
\LL \pi^* \DB(\pi(E)) \subset B_m \subset \cdots \subset B_1 \subset B_0, 
\end{equation*}
then $X$ admits a categorical weakly crepant resolution of singularities. 

\bigskip

\noindent Assume moreover that $B_m = \cdots = B_1 = B_0$, then $X$ admits a categorical strongly crepant resolution of singularities.
\end{theo} 

As a consequence, Kuznetsov obtains (see \cite{kuz}, sections $7$ and $8$) the:

\begin{cor} \label{kuzexample}
The following varieties admit a categorical strongly crepant resolution of singularities:
\begin{itemize}
\item a cone over $v_2(\mathbb{P}^n) \subset \mathbb{P}(S^2 \mathbb{C}^{n+1})$ (odd $n$),
\item a cone over $\mathbb{P}^n \times \mathbb{P}^n \subset \mathbb{P}(\mathbb{C}^{n+1} \ot \mathbb{C}^{n+1})$ (any $n$),
\item the Pfaffian variety : $\mathbb{P}f_4(n) := \mathbb{P} \{ \w \in \bigwedge^2 \mathbb{C}^n,\,\, \text{such that}\,\ rk(w) \leq 4 \}$ (odd $n$). 
\end{itemize}
\bigskip

\noindent The following varieties admit categorical weakly crepant resolution of singularities:
\begin{itemize}
\item a cone over a smooth Fano variety in its anti-canonical embedding,
\item the Pfaffian variety $\mathbb{P}f_4(n)$ (even $n$).
\end{itemize}
\end{cor}
Of course, one would like to generalize Kuznetsov's result, to apply it to higher corank determinantal varieties for instance. Using Kodaira relative vanishing theorem and some adjunction formulae, it is not difficult to prove the following (this is the case $n=1$ of Proposition \ref{keypropbis}):

\begin{prop} 
Let $X$ be an algebraic variety with Gorenstein rational singularities. Let $\pi : \tilde{X} \rightarrow X$ be a resolution of singularities such that the exceptional divisor $E$ of $\pi$ is irreducible, smooth and flat over $\pi(E)$. Then there exists a positive integer $m$ such that:

\begin{equation*}
K_{\tilde{X}} = K_X \ot \OO_{\tilde{X}}(mE),
\end{equation*}
and we have a semi-orthogonal decomposition:
\begin{equation*}
\DB(E) = \langle \OO_E(mE) \ot B_m, \ldots, \OO_E(E) \ot B_1, B_0 \rangle,
\end{equation*}
with $\LL \pi^* \DP(\pi(E)) \subset B_m \subset \cdots \subset B_1 \subset B_0$.
\end{prop}
 As a consequence of this proposition, we get a first mild generalization of the first part of Kuznetsov's theorem:
 
\begin{theo} \label{generalmild}
Let $X$ be an algebraic variety with Gorenstein rational singularities. Let $\pi : \tilde{X} \rightarrow X$ be a resolution of singularities such that the exceptional divisor $E$ of $\pi$ is irreducible, smooth and flat over $\pi(E)$. Then $X$ admits a categorical weakly crepant resolution of singularities.
\end{theo}
Note that all examples in Corollary \ref{kuzexample} satisfy the hypotheses of Theorem \ref{generalmild}.

\bigskip

\noindent Now, one remembers the strong version of Hironaka's theorem (\cite{hiro}) : any variety can be desingularized by a sequence of blow-ups such that every exceptional divisor is flat over its center of blowing-up and the total exceptional divisor of the resolution has (with its \underline{reduced} structure) simple normal crossings. So, one could hope to get a very far-reaching generalization of Kuznetsov's result : \textit{any Gorenstein variety with rational singularities admits a categorical weakly crepant resolution of singularities}.

\bigskip

Unfortunately, things are not so simple. Indeed, in order to construct a categorical crepant resolution starting from a sequence of blow-ups which desingularizes our variety, one needs strong compatibility conditions between the semi-orthogonal decompositions of the derived categories of the various exceptional divisors. Those compatibility conditions can be formulated at the categorical level (see Proposition \ref{keypropbis}). But one would rather like to know geometric situations where these compatibility conditions are satisfied. I have thus formalized the notion of \emph{wonderful resolution of singularities} (see definition \ref{wonderful} of the present paper), which applies to a sequence of blow-ups:
\begin{equation*}
X_n \rightarrow \cdots \rightarrow X_1 \rightarrow X,
\end{equation*}
giving a resolution of singularities of $X$. It is remarkable that this definition, which I made up to describe geometrically some compatibility conditions among the derived categories of the exceptional divisors of a resolution of singularities, happens to be the one which perfectly identifies the resolution process of the boundary divisor for the most basic wonderful compactifications of semi-simple linear algebraic groups (see \cite{procesi-decon} for details). With this notion in hand, quite technical but predictable computations yields the:

\begin{theo}[Main Theorem]
Let $X$ be an algebraic variety with Gorenstein rational singularities. Assume that $X$ has a wonderful resolution of singularities. Then $X$ admits a categorical weakly crepant resolution of singularities.
\end{theo}
At first glance, the notion of wonderful resolution of singularities seems to be quite restrictive and one could naively guess that there are too few examples of such resolutions. However, some reformulations of the work of Vainsencher \cite{vain} and Thaddeus \cite{thad} show that it is not the case. Indeed, we have the:

\begin{theo}[\cite{vain}, \cite{thad}]
All determinantal varieties (square as well as symmetric and skew-symmetric) admit wonderful resolution of singularities.
\end{theo}
As a consequence, we get the:

\begin{cor}
All Gorenstein determinantal varieties (square as well as symmetric and skew-symmetric) admit categorical weakly crepant resolutions of singularities.
\end{cor}

Let us now briefly indicate the plan of the paper. In section $2$, we give the definition of a wonderful resolution of singularities and study its basic cohomological properties. We also exhibit some examples of varieties which have a wonderful resolution of singularities. In section $3$, we prove the main theorem. This is the technical core of the paper. In section $4$, we discuss some minimality properties for categorical crepant resolutions of singularities and some existence problems related to prehomogeneous spaces.

\bigskip
\bigskip

\textbf{Aknowledgements} : I would like to thank Sasha Kuznetsov for many interesting discussions on
categorical resolutions of singularities and Christian Lehn for many helpful
comments on the first drafts of this paper. I would also like to thank Laurent
Manivel for his constant support and insightful criticism during the
preparation of this work.

\end{section}

\begin{section}{Wonderful resolutions of singularities}

We work over $\mathbb{C}$ the field of complex numbers. An
\textbf{algebraic variety} is a reduced algebraic scheme of finite type over
$\mathbb{C}$ (in particular it may be reducible). For any proper
morphism $f : X \rightarrow Y$ of schemes of finite type over $\mathbb{C}$, we
denote by $f_*$ the total derived functor $\RR f_* : \DB (X) \rightarrow \DB
(Y) $, by $f^*$ the total derived functor $\LL f^* : \mathrm{D^{-}}(Y)
\rightarrow \mathrm{D^{-}}(X)$ and by $f^!$ the right adjoint functor to $
\RR f_*$. In case we need to use
specific homology sheaves of these functors, we will denote them by $\Ri f_*,
\Li
f^*$ and $\Li f^!$.

 \begin{subsection}{Wonderful resolutions}
 
Let $Y \subset X$ be a closed
irreducible subvariety of $X$. We say that $Y$ is \textbf{ a
normally flat center in X} if the natural map:

$$ E \rightarrow Y $$
is flat, where $E$ is the exceptional divisor of the blow up of $X$ along $Y$.
Hironaka proved in \cite{hiro} that any algebraic variety can be desingularized
by a finite sequence of blow-ups along smooth normally flat centers.

\begin{exem} \upshape{ Let $W$ be a vector space of dimension at least $6$ and
let $X = \mathbb{P}(\{ \omega \in \bigwedge^{2}W,\, \operatorname{rank}(\omega)
\leq 4 \}$ be the $2$nd Pfaffian variety in $\bigwedge^{2}W$. Then $X$ is singular exactly along $\mathrm{Gr}(2,W) = \mathbb{P}(\{ \omega \in
\bigwedge^{2}W,\, \text{rank}(\omega) \leq 2 \}$. Here $\mathrm{Gr}(2,W)$ is a
smooth normally flat center for X. Indeed, if $E$ is the exceptional divisor of
the blow-up of $X$ along $\mathrm{Gr}(2,W)$, then $E$ is the flag variety
$\mathrm{Fl}(2,4,W)$ and the natural map onto $\mathrm{Gr}(2,W)$ is given by the
second projection. Obviously it is flat. See \cite{kuz}, section 8 for more
details.}
 
\end{exem}

Given $\mathrm{G}$ a semi-simple affine algebraic group, one often wants to find
a good equivariant compactification of $\mathrm{G}$. Equivariant
compactifications of $\mathrm{G}$ for
which the boundary
divisors have simple normal crossings are called \textit{wonderful} in the
literature (see \cite{hurug}, chapter $3$ for instance). One notices that the most basic wonderful compactifications we know are obtained by the following
procedure. Take $\overline{\mathrm{G}}$ be a naive \footnote{For instance, let $V$ be a linear representation of $\mathrm{G}$ and consider the closure of $\mathrm{G}$ in $\mathbb{P}(\mathrm{End}(V))$.} equivariant
compactification of $\mathrm{G}$. Then find an embedded resolution of the boundary divisor in
$\overline{\mathrm{G}}$ such that it's smooth model
has simple normal crossings with the exceptional divisors of the modification
of $\overline{\mathrm{G}}$. This embedded resolution is obtained by a
succession of blow-ups along smooth centers which satisfy nice intersection
properties. The following definition captures the most essential features of
this sequence of blow-ups.

\begin{defi}[Wonderful resolutions] 
\label{wonderful}

Let $X$ be an algebraic variety with Gorenstein singularities. For all
$n \geq 1$, we define a $n$-step \textbf{wonderful resolution of singularities}
in the following recursive way:

\begin{itemize}
 \item A $1$-step wonderful resolution is a single blow up:
 $$ \pi : \tilde{X} \rightarrow X,$$
 over a smooth normally flat center $Y \subset X$, such that $\tilde{X}$ and
the exceptional divisor $E \subset \tilde{X}$ are smooth.

\item For $n \geq 2$, a $n$-step wonderful resolution of X is a sequence of
blow-ups:
$$ X_n \stackrel{\pi_n}\rightarrow X_{n-1} \ldots
X_1 \stackrel{\pi_1}\rightarrow X_0 = X $$ over smooth normally flat centers
$Y_{k+1} \subset X_k$ such that:

\begin{enumerate}
 \item all the $X_k$ are Gorenstein,
 
 \item the map $\pi_2 \ldots \pi_n: X_n \rightarrow X_1$ is a $(n-1)$-step
wonderful resolution of $X_1$,

 \item the intersection of $Y_{k+1}$ with $E^{(k)}_1$ is proper and
smooth (where
$E^{(k)}_1$ is the total transform of
$E_1$, the exceptional divisor of $\pi_1$, with respect to $\pi_{2} \ldots
\pi_{k}$, for $k > 1$),

 \item the map $\pi_2 \ldots \pi_n|_{E^{(n)}_1} : E^{(n)}_1 \longrightarrow
E_1$ is a $(n-1)$-step wonderful resolution of singularities.

\end{enumerate}

\end{itemize}

\end{defi}

As far as I know, the term \textit{wonderful resolution} first appeared in
\cite{fu-chap} where it was used to describe the resolution of indeterminacies
of a stratified Mukai flop.

\begin{exem}[Determinantal varieties]
\label{coli}

\upshape{Let $E$ a vector space of dimension $n$. Let $2 \leq r \leq n$
and let $X$ be the subvariety of $\mathbb{P}(\mathrm{End}(E))$ defined by
the vanishing of the minors of size $r$. It is well known that $X$ is Gorenstein
with rational singularities (see \cite{weyman}, corollary $6.1.5$). We define
$X_1$ to be the blow up:
\begin{equation*}
 \pi_1 : X_1 \rightarrow X,
\end{equation*}
of $X$ along $Y_1$, where $Y_1$ is the subvariety of
$\mathbb{P}(\mathrm{End}(E))$ defined by the minors of size $2$. For $3 \leq k
\leq r-1$, we define recursively $X_k$ to be the blow-up:

\begin{equation*}
 \pi_k : X_k \rightarrow X_{k-1},
\end{equation*}
of $X_{k-1}$ along $Y_k$, where $Y_k$ is the strict transform through $\pi_1
\ldots \pi_{k-1}$ of the subvariety of $\mathbb{P}(\mathrm{End}(E))$ defined by
the vanishing of the minors of rank $k$. By theorem $1$ of \cite{vain}, we know that $X_{r-1}$ is smooth and
that all $Y_k$ for $1 \leq k \leq r-1$ are smooth. Moreover, theorem $2.4$ of
\cite{vain} shows that the $Y_k$ are
normally flat centers and that item $1$, $2$, $3$ and $4$ in the definition of a
wonderful resolution are satisfied for the following resolution of $X$:

\begin{equation*}
 X_{r-1} \stackrel{\pi_{r-1}}\rightarrow  \ldots
X_1 \stackrel{\pi_1}\rightarrow X
\end{equation*}
Thus, $X$ has Gorenstein rational singularities and admits a wonderful
resolution of singularities.

Let $X_r^{[sym]} \subset
\mathbb{P}(\mathrm{S}^2(E))$ (resp. $X_r^{[skew]} \subset
\mathbb{P}(\bigwedge^2(E))$ denotes the determinantal variety defined by the
vanishing of the minors of size $r$ of the generic $n \times n$ symmetric (resp.
skew-symmetric) matrix with linear entries. If $n-r$ is even (resp. no
conditions), we know by \cite{weyman}, corollary $6.3.7$ (resp. proposition
$6.4.3$) that $X_r^{[sym]}$ (resp. $X_r^{[skew]}$) is Gorenstein with rational
singularities. Moreover, the appendices $10$ and $11$ of
\cite{thad} show that $X_r^{[sym]}$ and $X_r^{[skew]}$ also admit wonderful
resolutions of singularities.}
 
 \end{exem}

\begin{exem}[Secant variety of $\mathbb{OP}^2$]
\upshape{
Let $X = \mathbb{OP}^2 = \mathrm{E_6/P_{\alpha_1}} \subset \mathbb{P}^{26}$ be
the Cayley
plane into his highest weight embedding, where $\mathrm{P_{\alpha_1}}$ is the
maximal parabolic associated to the root $\mathrm{\alpha_1}$ of the
$\mathrm{E_6}$ root system. The Cayley plane can also be recovered
as the scheme defined by the $2 \times 2$ minors of the generic hermitian $3
\times 3$ octonionic matrix:

\begin{equation*}
M = \begin{pmatrix}
a_1 & a_2 & a_3 \\
\overline{a_2} & a_4 & a_5 \\
\overline{a_3} & \overline{a_5} & a_6 \\
 \end{pmatrix}, \text{with}\,\, a_i \in \mathbb{O} \, \text{and
such that} \,\, \leftexp{t}{\overline{M}} = M, 
\end{equation*}
where $\overline{a_i}$ is the octonionic conjugate of $a_i \in
\mathbb{O}$.
 
Let $S(X)$ be the secant variety of $X$ inside $\mathbb{P}^{26}$. It can be
seen as the scheme defined by the determinant of the above matrix $M$. The
variety $S(X)$ is a cubic hypersurface (so it is Gorenstein) which is singular
exactly along $\mathbb{OP}^2$. Let
\begin{equation*}
\pi : \widetilde{S(X)} \rightarrow S(X) 
\end{equation*}
be the blow up of $S(X)$ along $X$. It is smooth and the exceptional divisor
$E$ is isomorphic to $\mathrm{E_6/Q_{\alpha_1, \alpha_5}}$ (where
$\mathrm{Q_{\alpha_1, \alpha_5}}$ is the parabolic associated to the roots
$\mathrm{\alpha_1}$ and $\mathrm{\alpha_5}$). It is a fibration into smooth
$8$-dimensional quadrics over $X$. As a consequence, the map $\pi :
\widetilde{S(X)} \rightarrow S(X)$ is a
wonderful resolution of singularities. We refer to \cite{zak} ch. III and
\cite{manifaen} for more details on the beautiful geometric and categorical
features of the Cayley plane.
}
\end{exem}

\end{subsection}

\begin{subsection}{Wonderful resolutions and singularities of the intermediate
divisors}

The above examples also suggest that the definition of a wonderful resolution
imposes strong conditions on the singularities of the
exceptional divisor $E_{1}^{(k)}$. Indeed, the following three propositions
show that they must be similar to the singularities of $X$.

\begin{prop} \label{rational} Let $X$ be an algebraic variety with Gorenstein
and rational
singularities. Let
$X_n \stackrel{\pi_n}\rightarrow X_{n-1} \ldots X_1 \stackrel{\pi_1}\rightarrow
X_0 = X$ be a wonderful resolution of singularities of $X$. Then, for all $1
\leq
k \leq n$, the varieties $X_k$ and the exceptional divisors $E_1^{(k)}$ have
Gorenstein rational singularities. 

\end{prop}

\begin{proof} The fact that the $X_k$ are Gorenstein is in the
definition of a wonderful resolution of singularities. Let $\pi_1 : X_1
\rightarrow X$ be the blow-up along $X_1$ with exceptional divisor $E_1$ and
let $E_{\alpha_1},\ldots, E_{\alpha_d}$ be the irreducible components of $E_1$.
We also denote by $E_{\alpha_i}^{(k)}$ the total transform of $E_{\alpha_i}$
under $\pi_2 \ldots \pi_k$. Since the $Y_k$ meet transversally the
$E_{\alpha_i}^{(k)}$ for all $ \geq 2$, the coefficients appearing in front of
the $E_{\alpha_i}^{(k)}$ in the expression of $K_{X_k}$ are the same as the
coefficients in front of the $E_{\alpha_i}$ in the expression of $K_{X_1}$. Now
$X$ is Gorenstein with rational singularities, hence it has canonical
singularities by \cite{kollar} corollary $11.13$. Thus, the coefficients in
front of the $E_{\alpha_i}^{(k)}$ in the expression of $K_{X_k}$ are positive,
and so are the coefficient in front of the $E_{\alpha_i}$ in the expression of
$K_{X_1}$. As a consequence, we have ${\pi_1}_* \OO_{X_1} = \OO_{X}$. Since we
also have ${\pi_1 \ldots \pi_n}_* \OO_{X_n} = \OO_{X}$, we find that $X_1$ has
rational singularities. An obvious induction shows that all $X_k$ have
rational singularities.

The divisors $E_1^{(k)}$ are obviously Gorenstein, as they are Cartier divisors
inside Gorenstein varieties. The main point of the proposition is thus to show
that the $E_1^{(k)}$ have rational singularities for any $k \geq 1$.

\bigskip

Item $3$ of definition \ref{wonderful} implies that
for any $k \geq 2$ the map
\begin{equation*}
 \pi_k|_{E_1^{(k)}} : E_1^{(k)} \rightarrow E_1^{(k-1)}
\end{equation*}
is the blow-up along $Y_{k} \cap E_1^{(k-1)}$ (we recall that $E_1^{(k)}$ is
the total transform of $E_1$ through $\pi_2 \ldots \pi_k$). Thus, we deduce
from item $4$ of definition \ref{wonderful} that for any $k \geq 1$, the map

\begin{equation*}
 \Pi_{k} = \pi_{k+1} \ldots \pi_{n}|_{E_1^{(n)}} : E_1^{(n)} \rightarrow
E_1^{(k)}
\end{equation*}
is a resolution of singularities. Hence, to prove that $E_1^{(k)}$ has rational
singularities we only have to compute $(\Pi_{k})_* \OO_{E_1^{(n)}}$.

For all $1 \leq k \leq n$, we have a fibered diagram:
\begin{equation*}
 \xymatrix{
 E_1^{(n)} \ar[rr]^{i_1^{(n)}} \ar[dd]_{\Pi_k} & & X_n \ar[dd]^{\pi_{k+1} \ldots
\pi_n} \\
& & \\
E_1^{(k)} \ar[rr]^{i_1^{(k)}}& & X_k}
\end{equation*}

The maps $i_1^{(n)}$ and $i_1^{(k)}$ are locally complete intersection
embeddings
and $\operatorname{codim} E_1^{(k)} \subset X_k = \operatorname{codim} E_1^{(n)}
\subset X_n$ so that we have the commutation of derived functors (see
\cite{kuz4}, corollary
$2.27$):
$$  (i_1^{(k)})^*  (\pi_{q+1} \ldots \pi_n)_* \mathcal{O}_{X_n} = 
(\Pi_k)_* (i_1^{(n)})^* \mathcal{O}_{X_n},$$
that is:
$$  (\Pi_k)_* \mathcal{O}_{E_1^{(n)}} = \mathcal{O}_{E_1^{(k)}}.$$ So
$E_1^{(k)}$ has rational singularities.

\end{proof}

The next proposition (\ref{connexe}) is important, because for any wonderful resolution of a variety $X$ with rational Gorenstein singularities:
\begin{equation*}
X_n \stackrel{\pi_n}\rightarrow X_{n-1} \ldots X_1
\stackrel{\pi_1}\rightarrow X_0 = X,
\end{equation*}
it allows to produce a simple formula relating $K_{X_k}, \pi_{k}^*K_{X_{k-1}}$
and $E_k$ for any $1 \leq k \leq n$.  The (positive) coefficient appearing
in front of $E_k$ in the expression of $K_{X_k}$ can be interpreted as a kind of multiplicity of $Y_k$ in $X_{k-1}$. Proposition \ref{canonical} then shows that the multiplicity of $Y_{k+1}$ in $X_{k}$ is equal to the multiplicity of
$Y_{k+1} \cap E_1^{(k)}$ in $E_1^{(k)}$ for any $k \geq 1$.

\begin{prop}
\label{connexe}
Let $X$ be an algebraic variety with Gorenstein and rational
singularities. Let
$X_n \stackrel{\pi_n}\rightarrow X_{n-1} \ldots X_1 \stackrel{\pi_1}\rightarrow
X_0 = X$ be a wonderful resolution of singularities of $X$. Then the
exceptional
divisor $E_1$ of $\pi_1 : X_1 \rightarrow X_0$ is irreducible.
\end{prop}

\begin{proof} For $y \in Y_1$, denote by $\mathcal{C}_{y}(X)$ the tangent cone
to $X$ at $y$ and let $\rho_Y : \mathcal{C}_{Y_1}(X) \rightarrow Y_1$ be the
tangent
cone to $X$ along $Y_1$. Since $Y_1$ is a smooth normally flat center in $X$, we
know by \cite{hiro}, theorem 2, p.195, that $\mathcal{C}_{y}(X)$ is a
cone with vertex $T_{Y_1,y}$ over $\rho_{Y_1}^{-1}(y)$. Since $X$ is
Cohen-Macaulay,
the main result of \cite{schaub} implies that the cone $\mathcal{C}_{y}(X)$
is connected in codimension $1$ (which means that one needs to subtract from
$\mathcal{C}_{y}(X)$ a variety of codimension at most $1$ to disconnect it).
As $\dim \rho_{Y_1}^{-1}(y) \geq 1$ for all $y \in Y$, we get that  $\rho_{Y_1}^{-1}(y)$ is connected in codimension $1$ for all $y
\in Y$. The map $\rho_{Y_1}$ being flat, this finally implies that $E_1 =
\mathbb{P}(\mathcal{C}_{Y_1}(X))$ is connected in codimension $1$.

Assume that $E_1$ is reducible. By the above discussion, we know that each
irreducible component of $E_1$ meets another component in codimension $1$. Since $\dim E_1 \geq 2$, we get that $E_1$ is not normal, which contradicts the fact that it has rational singularities (see \ref{rational}).
\end{proof}

\bigskip

In the following, for $2 \leq k \leq n$, we denote by $E_{1k}^{(k)}$ the
intersection
of $E_1^{(k)}$ with $E_k$, the exceptional divisor of $\pi_k : X_k \rightarrow
X_{k-1}$. By item $3$ of definition \ref{wonderful}, $E_{1k}^{(k)}$ is also the
exceptional divisor of the blow-up
$ E_1^{(k)} \rightarrow E_1^{(k-1)}$.

\begin{prop} \label{canonical} Let $2\leq k\leq n$. If $m_k$ is the non-negative
integer such
that $K_{X_{k}} = \pi_{k}^{*}K_{X_{k-1}} \ot \OO_{X_k}(m_k E_k)$, then we
have:
\begin{equation*}
K_{E^{(k)}_1} = \pi_k|_{E_1^{(k)}}^{*}K_{E_1^{(k-1)}} \ot \OO_{E_1^{(k)}}(m_{k}
E_{1k}^{(k)}).
\end{equation*}
\end{prop}

Note that the integer $m_k$ is well defined because $X_k \rightarrow X_{k-1}$
is a blow-up whose exceptional divisor is irreducible. Moreover, it is non-negative because Gorenstein rational singularities are canonical (see \cite{kollar}, corollary $11.13$).

\begin{proof}

The adjunction formula implies that $ K_{E_1^{(k)}} = K_{X_k}|_{E_1^{(k)}} \ot
\mathcal{O}_{E_1^{(k)}}(E_1^{(k)})$ and that $K_{E_1^{(k-1)}} =
K_{X_{k-1}}|_{E_1^{(k-1)}} \ot \mathcal{O}_{E_1^{(k-1)}}(E_1^{(k-1)})$.

Now, we tensor the formula $K_{X_k} = \pi_k^* K_{X_{k-1}} \ot \OO_{X_k}(m_k
E_k)$ by
$\mathcal{O}_{E_1^{(k)}}$ and taking into account the fact that
$E_k|_{E_1^{(k)}} = E_{1k}^k$ and that ${\pi_k}|_{E_1^{(k)}}^*
\mathcal{O}_{E_1^{(k-1)}}(E_1^{(k-1)}) = \mathcal{O}_{E_1^{(k)}}(E_1^{(k)})$, we
get the announced formula.

\end{proof}

\end{subsection}
\end{section}

\begin{section}{Categorical crepant resolutions of singularities and wonderful resolutions}
\begin{subsection}{Categorical crepant resolution of singularities}
 
Now we come back to the notion of categorical resolutions of
singularities. Let us recall some basic facts about derived categories of coherent sheaves on an algebraic variety.

\bigskip

Let $X$ be an algebraic variety, then $\DB(X)$ denotes the derived category of
bounded complexes of coherent sheaves on $X$. The subcategory of bounded
complexes of
locally free sheaves is denoted by $\DP(X)$. Recall that if $X$ is
smooth, then these two categories are equivalent. Let $A$ be a full subcategory
of $\DB(X)$ and $\alpha : A \hookrightarrow \DB(X)$ the embedding functor. We
say that $A$ is \textbf{admissible} if $\alpha$ has a left and a right adjoint. This is equivalent to asking that there exist semi-orthogonal decompositions:

\begin{equation*}
 \DB(X) = \langle A, \leftexp{\perp}{A} \rangle \,\, \text{and} \,\, \DB(X) =
\langle A^{\perp}, A \rangle,
\end{equation*}

where: 

\begin{equation*}
 \leftexp{\perp}{A} = \{ F \in \DB(X),\,\, \HH(F,a) = 0,\,\, \text{for all} \,
a \in A \},
\end{equation*}

and

\begin{equation*}
 A^{\perp} =  \{ F \in \DB(X),\,\, \HH(a,F) = 0,\,\, \text{for all} \,
a \in A \}.
\end{equation*}

We refer to \cite{kuz}, section $2$ for more details on semi-orthogonal
decompositions. Recall the definition of a categorical crepant resolution of singularities we gave in the introduction (see also \cite{kuz}):

\begin{defi}
Let $X$ be an algebraic variety with Gorenstein and rational singularities. A
\emph{categorical resolution of singularities} of $X$ is a
triangulated category $\T$ with a functor $\RR {\pi_{\T}}_* : \T \rightarrow
\DB(X)$ such that:
\begin{itemize}
\item there exists a resolution of singularities $\pi : \tilde{X} \rightarrow
X$ such that $\d : \T \hookrightarrow \DB(\tilde{X})$ is admissible and $\RR
{\pi_{\T}}_* = \RR \pi_* \circ \d$,

\item we have $\LL \pi^* \DP(X) \subset \T$ and for all $\F \in \DP(X)$:

\begin{equation*}
\RR {\pi_{\T}}_* \LL \pi_{\T}^* \F \simeq \F,
\end{equation*}
where $\LL \pi_{\T}^*$ is the left adjoint to $\RR {\pi_{\T}}_*$.
\end{itemize}
\bigskip

\noindent If for all $\F \in \DP(X)$, there is a quasi-isomorphism:
\begin{equation*}
\LL \pi_{\T}^* \F \simeq \LL \pi_{\T}^! \F,
\end{equation*}
where $\LL \pi_{\T}^!$ is the right adjoint of $\RR {\pi_{\T}}_*$, we say that
$\T$ is \emph{weakly crepant}.
\end{defi}

Now we can state our main theorem.

\begin{theo}[Existence of categorical weakly crepant resolutions]
\label{maintheo}

Let $X$ be an algebraic variety with Gorenstein
and rational
singularities. Assume that $X$ has a wonderful
resolution of singularities. Then $X$ admits a categorical weakly crepant
resolution of singularities.
 
\end{theo}

An analogous notion of "non commutative" crepant resolution of singularities was also studied for determinantal varieties in \cite{blv}, \cite{blvII} and \cite{weyII}. In these works, non-commutative crepant resolutions are proved to exist for some determinantal varieties which already admit "geometric" crepant resolution of singularities.

\end{subsection}
\begin{subsection}{Lefschetz semi-orthogonal decompositions}
  
A key point in the proof of the theorem is the notion of dual Lefschetz
decomposition which was introduced in \cite{kuz2}.

\begin{defi} Let $X$ be an algebraic variety and $L$ a line bundle on $X$. Let
$\T \subset \DB(X)$ be a full admissible subcategory such that for all $T \in
\T$, we have $T \ot L \in \T$. We say that $\T$ admits a \textbf{dual Lefschetz
decomposition} with respect to $L$, if there exists a semi-orthogonal
decomposition:

$$ \T= \langle B_m \ot L^{\ot m}, B_{m-1} \ot L^{\ot m-1},\ldots, B_0
\rangle,$$ where $B_m \subset B_{m-1} \subset \ldots \subset B_0$ are full
admissible
subcategories of $\T$.

\end{defi}

\begin{exem} \upshape{The following semi-orthogonal decomposition:
 
 $$ \DB(\mathbb{P}^n) = \langle
\mathcal{O}_{\mathbb{P}^n}(-n),
\mathcal{O}_{\mathbb{P}^n}(-n+1),\ldots,\mathcal{O}_{
\mathbb{P}^n} \rangle,$$ is a dual Lefschetz decomposition of
$\DB(\mathbb{P}^n)$ with respect to
$\mathcal{O}_{\mathbb{P}^n}(-1)$.}

\end{exem}

The following lemma is proposition $4.1$ of \cite{kuz}.

\begin{lem} 
\label{divisor}
Let $X$ be a smooth algebraic variety and let $ i : E \hookrightarrow X$ a
Cartier
divisor such that we have a dual Lefschetz decomposition:

$$ \DP (E) = \langle B_m \ot \OO_{E}(mE) \ldots B_1 \ot \OO_{E}(E),\, B_0
\rangle,$$ with $B_m \subset \ldots \subset B_1 \subset B_0$. Then we have a
semi-orthogonal decomposition:

$$ \DB (X) = \langle i_*(B_m \ot \OO_{E}(E)) \ldots i_*(B_1 \ot \OO_{E}(E)),\,
A_0(X) \rangle,$$ where $A_0(X) = \{ F \in \DB (X),\, i^*F \in B_0 \}$.

\end{lem}

The technical tool for the proof of the main theorem is the following
proposition. We first need some notation. Let $X$ be an algebraic variety with
Gorenstein and canonical singularities. Let

\begin{equation*}
X_n \stackrel{\pi_n}\rightarrow X_{n-1} \ldots
X_1 \stackrel{\pi_1}\rightarrow X_0 = X,
\end{equation*} be a wonderful resolution of $X$, which is a succession of
blow-ups along the
smooth normally flat centers $Y_t \subset X_{t-1}$ for $t=1 \ldots n$. For any
such $t$, we have a diagram:
\begin{equation*}
 \xymatrix{
 E_t^{(n)} \ar[rr]^{i_t^{(n)}} \ar[dd]_{q_t} & & X_n
\ar[dd]^{\pi_{t} \ldots
\pi_n} \\
& & \\
Y_t \ar[rr] & & X_{t-1}}
\end{equation*}
We recall that $E_t^{(n)}$ is the total transform of $E_t \subset X_t$ through
$\pi_{t+1}
\ldots \pi_n$ and that $q_t$ is a flat projection
with smooth fibers. We also have the formula $K_{X_t} = \pi_t^{*} K_{X_{t-1}}
+ m_t E_t$. For any $1 \leq j \leq n$ and for any $0 \leq k_j \leq m_j-1$, we
define the subcategories of $\DB (X_n)$:
\begin{equation*}
A_{j,k_j} = (i_j^{(n)})_* \left[ q_j^{*} \DB (Y_j) \ot
\OO_{E_j^{(n)}} \left( (m_j-k_j) E_j^{(n)}
+ \sum_{t=j+1}^{n} m_t E_t^{(n)} \right) \right],
\end{equation*}

\begin{prop} \label{keyprop} With the above
hypotheses and notations, we have a
semi-orthogonal decomposition:

\begin{equation*}
\DB (X_n) = \langle A_{1,0}, \ldots, A_{1,m_1-1}, \ldots, A_{j,k_j}, \ldots,
A_{n,m_n-1},\, D_{X_n} \rangle,
\end{equation*}
where $D_{X_n}$ is the left orthogonal to the full admissible subcategory
generated by the $A_{j,k_j}$ for $0 \leq k_j \leq m_j-1$ and $1 \leq j \leq n$. Moreover, we have the inclusion $(\pi_1
\ldots \pi_n)^{*} \DP(X) \subset D_{X_n}$.

\end{prop}

We postpone the proof of the proposition and we will show that it implies that
$D_{X_n}$ is a categorical weakly crepant resolution of singularities of $X$.

\bigskip

Let $\delta : D_{X_n} \rightarrow \DB (X_n)$ be the fully faithful admissible
embedding and denote by $\pi$ the resolution $X_n \rightarrow X$. Since $\pi^*
\DP (X) \subset D_{X_n}$, the only thing left to prove is the crepancy of $\d
\pi^*$, that is $\d^{*} \pi^{*}(F) = \d^{!} \pi^{!}(F)$, for all $F \in \DP
(X)$. Recall that
\begin{equation*}
\pi^!(F) = \pi^*(F) \ot \pi^*({K_X}^{-1})\ot K_{X_n} = \pi^*(F) \ot
\OO_{X_n}(\sum_{t=1}^{n} m_t E_t^{(n)}),
\end{equation*}
for any $F \in \DP(X)$.

Now, since the functor $\d$ is fully faithful, the equality $\d^{*} \pi^{*}(F) =
\d^{!} \pi^{!}(F)$ is equivalent to $\d(\d^*
\pi^*(F)) =
\d (\d^! \pi^!(F))$, for any $F \in \DP(X)$. As $\pi^* \DP (X) \subset D_{X_n}$,
we have $\d(\d^* \pi^*(F)) = \pi^*(F)$, for all $F \in \DP (X)$. We are going to
show that $\d (\d^! \pi^!(F)) = \pi^*(F)$ for any $F \in \DP(X)$. For $1
\leq j \leq n$ and $0 \leq k_j \leq m_j-1$, we have exact
sequences:

\begin{equation*}
\begin{split}
& 0 \rightarrow \OO_{X_n} \left((m_j-k_j-1)E_1^{(n)} + \sum_{t=j+1}^{n} m_t
E_t \right) \rightarrow  \OO_{X_n} \left((m_j-k_j)E_j^{(n)} + \sum_{t=j+1}^{n}
m_t E_t \right) \rightarrow \\
& (i_j^{(n)})_* \OO_{E_j^{(n)}} \left((m_j-k_j) E_j^{(n)} + \sum_{t=j+1}^{n}
m_tE_t \right) \rightarrow 0,\\
\end{split}
\end{equation*}

So for each $F \in \DP (X)$, we deduce a long sequence of triangles:

\begin{equation*}
\xymatrix@C=10pt{ 
\pi^*(F)  \ar[rr]  & & F_{n,m_n-1}
\ar[ld] \ar[r] & \ldots \ar[rr] & & F_{j,k_j}
\ar[ld] \ar[r] & \ldots \ar[r] &  F_{1,1} \ar[rr] & & F_{1,0} \ar[ld]  \\
  & \F_{n,m_n-1} \ar[lu] & \ldots  & &  \F_{j,k_j} \ar[lu] & \ldots &  &  &
\F_{1,0} \ar[lu] & \\
} 
\end{equation*}

where:

\begin{equation*}
F_{j,k_j} = \pi^* F \ot \OO_{X_n}\left((m_j-k_j) E_j^{(n)} + \sum_{t=j+1}^{n}
m_tE_t^{(n)} \right), 
\end{equation*}

and

\begin{equation*}
\F_{j,k_j} =  (i_j^{(n)})_*
(i_j^{(n)})^* \left[ \pi^* F \ot \OO_{X_n} \left((m_j-k_j) E_j^{(n)} +
\sum_{t=j+1}^{n}
m_t
E_t^{(n)} \right) \right], 
\end{equation*}

for $0 \leq k_j \leq m_j-1$ and $1 \leq j \leq n$.

\bigskip

Since $\d : D_{X_n} \hookrightarrow \DB (X_n)$ is fully faithful and
admissible, it is well known that $\d (\d^! \pi^!(F))$ is the
$D_{X_n}$-component of
$\pi^! (F) = \pi^*(F) \ot \OO_{X_n}(\sum_{t=1}^{n} m_t E_t^{(n)})$ in the
semi-orthogonal decomposition of $\DB (X_n)$ given by proposition \ref{keyprop}
(see \cite{kuz} section $2$ for more details on semi-orthogonal decompositions). Consider the fibered diagram:
\begin{equation*}
 \xymatrix{
 E_j^{(n)} \ar[rr]^{i_j^{(n)}} \ar[dd]_{q_j} & & X_n \ar[dd]^{\pi_{j} \ldots
\pi_n} \ar@/^6pc/[dddd]^{\pi} \\
& & \\
Y_j \ar[rr]^{i_j} & & X_{j-1} \ar[dd]^{\pi_1 \ldots \pi_{j-1}} \\
& & \\
 & & X}
\end{equation*}

Since $(i_j^{(n)})^{*}(\pi^*(F)) = q_j^*((\pi_1
\ldots \pi_{j-1} i_j)^{*}(F)) \in q_j^*(\DB (Y_j))$, we have
$\mathcal{F}_{j,k_j} \in A_{j,k_j}$, for $1 \leq j \leq n$ and $0 \leq k_j \leq
m_j-1$. As a consequence, the
above sequence of triangles shows that $\d(\d^{!} \pi^{!}(F)) = \pi^*(F)$ and we
are done.
 
\end{subsection}

\begin{subsection}{Some vanishing lemmas}

Before diving into the proof of proposition \ref{keyprop}, we need a
vanishing result which will be very useful.

\begin{lem} \label{vanishing} Let $f : Z \rightarrow S $ be a flat and
projective morphism such that $Z$ has Gorenstein rational
singularities. Assume that there exists a relatively anti-ample divisor $E
\subset Z$, a line
bundle $F$ on $S$ and a
positive integer $r \geq 1$ such that:
$$ K_Z = f^*F \ot \OO_{Z}((r+1)E).$$

Then we have the vanishing:

$$\Ri f_* \mathcal{O}_{Z}(kE) = 0$$
for any $i > 0$ if $k \leq r$ 
and for any $i < \dim Z - \dim S$ if $k \geq 1$.
As a consequence, we have:
$$ \Hi (X,f^*A \ot \mathcal{O}_{X}(kE)) = 0$$
for any vector bundle $A$ on $S$ for any $i \geq 0$ if $1 \leq k \leq r$.
\end{lem}

\begin{proof} By hypotheses, we have $K_{Z} = f^*F \ot \OO_{Z}((r+1)E)$, the
variety $Z$ has rational singularities and $E$ is relatively
anti-ample so that by Kawamata-Viehweg relative vanishing (see \cite{KMM},
theorem $1.2.5$) we have
$$\Ri f_* \mathcal{O}_{Z}(kE) = 0$$
for any $i > 0$ if $k \leq r$. 

Again by hypotheses, we know that $Z$ is Gorenstein and that $f$ is flat.
This implies that all the fibers are pure $d$-dimensional
Gorenstein schemes, where $d = \dim Z - \dim S$. As a consequence, we can apply
the relative duality for a flat morphism with Gorenstein fibers
(see \cite{kleiman}, theorem $20$) and we find

\begin{equation*}
\Ri f_* (\OO_{Z}(-kE) \ot K_{Z/S}) =  \Hh_{\OO_S} (\mathrm{R^{d-i}}f_*
\OO_{Z}(kE), \OO_{S}).
\end{equation*}
As a consequence, we have
\begin{equation*}
\Ri f_* \OO_{Z}(kE) = 0,
\end{equation*}
for any $i < d$ if $k \geq 1$.

The last vanishing
$$\Hi (X,f^*A \ot \mathcal{O}_{X}(kE)) = 0$$ for any vector bundle $A$ on $S$
and any $i \geq 0$ if $1 \leq k \leq r$ is then a direct consequence of the
projection formula and Leray's spectral sequence.
\end{proof}

\bigskip

Thus, we get the following corollary:

\begin{cor}
\label{corvanishing}
Let $X$ be an algebraic variety with Gorenstein singularities and let $f :
\tilde{X} \rightarrow X$ be the blow-up of $X$ along the smooth normally flat
center $Y \subset X$ with exceptional divisor $E$. Assume that $E$ has
Gorenstein rational singularities and that $K_{\tilde{X}} =
f^*K_X + rE$ for a positive integer $r$. Then we have:
\begin{equation*}
 \Ri f_* \OO_{\tilde{X}}(kE) = 0
 \end{equation*}
for all $i > 0$ if $k \leq r$ and
\begin{equation*}
\R0 f_* \OO_{\tilde{X}}(kE) = \OO_{X} 
\end{equation*}
if $k \geq 0$.
\end{cor}

\begin{proof}
We will proceed by induction. The divisor $-E$ is relatively ample
so by Grothendieck's vanishing theorem we have $\Ri f_* \OO_{\tilde{X}}(kE) =
0$ for all $i > 0$, if $k << 0$. Now, let $k \leq r-1$ be an integer such that
$\Ri f_* \OO_{\tilde{X}}(kE) = 0$ for all $i > 0$. We have the exact sequence:

\begin{equation*}
0 \rightarrow \OO_{\tilde{X}}(k E) \rightarrow \OO_{\tilde{X}}((k +1)E)
\rightarrow \OO_{E}((k+1)E) \rightarrow 0, 
\end{equation*}
and by the adjunction formula : $K_E = f^*|_{E} K_X + (r+1) E$. So, the
long exact sequence in cohomology associated to $f_*$ and lemma \ref{vanishing}
applied to $f :  E \rightarrow Y$ show that $\Ri f_* \OO_{\tilde{X}}((k+1)E) =
0$ for any $i > 0$.
 
\bigskip
 
The variety $X$ is Gorenstein, hence Cohen Macaulay, so by Schaub's theorem
\cite{schaub} (see the proof of proposition \ref{connexe}), we know that $\R0
f_*\OO_{\tilde{X}} = \OO_{X}$. Using the above exact sequence and lemma
\ref{vanishing}, we prove by induction that $\R0 f_* \OO_{\tilde{X}}(kE) =
\OO_{X}$ for any $k \geq 0$ 
\end{proof}

\bigskip

One of course notes that part of corollary \ref{corvanishing} is a direct
consequence of Kawamata-Viehweg relative vanishing. But in any case we will need
both proposition \ref{vanishing} and corollary \ref{corvanishing} in the proof
of the main theorem.

\end{subsection}

\begin{subsection}{The technical induction}

The proof of proposition \ref{keyprop} will be done by induction, but we will
have to prove a more precise statement. We begin with some more notation. Let
$X$ be an algebraic variety with Gorenstein and canonical singularities. Let

\begin{equation*}
X_n \stackrel{\pi_n}\rightarrow X_{n-1} \ldots
X_1 \stackrel{\pi_1}\rightarrow X_0 = X,
\end{equation*}
be a wonderful resolution of $X$. We introduced the following subcategories of
$\DB (X_n)$:

\begin{equation*}
A_{j,k_j} = {i_j^{(n)}}_*\left[q_j^{*} \DB (Y_j) \ot
\OO_{E_j^{(n)}} \left((m_j-k_j)
E_j^{(n)}
+ \sum_{t=j+1}^{n} m_t E_t^{(n)} \right) \right],
\end{equation*}

for all  $0 \leq k_j \leq m_j-1$ and $1 \leq j \leq n$.
 
\bigskip

Now, let $1 \leq k \leq n$, and let $k+1 \leq p \leq n$. We denote by
$j_{k,p}^{(n)} :
E_{k,p}^{(n)} \hookrightarrow E_k^{(n)}$ the embedding of the intersection
$E_{k,p}^{(n)} = E_p^{(n)} \cap E_k^{(n)}$ which is also the total transform of
$E_{k,p}^{(p)} = E_p \cap E_k^{(p)}$ through $\pi_{p+1} \ldots \pi_n$, by item
$3$ and $4$ of definition \ref{wonderful}.
Hence we have a fibered diagram:

\begin{equation*}
 \xymatrix{
E_{k,p}^{(n)} = E_p^{(n)} \cap E_k^{(n)} \ar[rr] \ar[dd] & &
X_n \ar[dd]^{\pi_{p+1} \ldots
\pi_n} \\
& & \\
E_{k,p}^{(p)} = E_p \cap E_k^{(p)} \ar[rr]& & X_k}
\end{equation*}

We denote by $q_{k,p}$ the flat map $q_{k,p} : E_{k,p}^{(n)} \rightarrow
Y_{k,p}$,
where $Y_{k,p}$ is the smooth and proper intersection $ Y_{k,p} = Y_p \cap
E_{k}^{(p-1)}$. Notice that, again by item $3$ and $4$ of definition
\ref{wonderful}, the intersection $E_p \cap E_k^{(p)}$ is the exceptional
divisor of the blow up of $E_k^{(p-1)}$ along $Y_{k,p}$. Thus, we have another
fibered diagram: 

\begin{equation*}
 \xymatrix{
E_{k,p}^{(n)} \ar[rr]^{j_{k,p}^{(n)}} \ar[dd] \ar@/_6pc/[dddd]_{q_{k,p}} & &
E_{k}^{(n)} \ar[dd]^{\pi_{p+1} \ldots \pi_n|_{E_k^{(n)}}} \\
& & \\
E_{k,p}^{(p)} \ar[dd] \ar[rr]^{j_{k,p}^{(p)}}& & E_k^{(p)}
\ar[dd]^{\pi_p|_{E_k^{(p)}}}\\
& & \\
Y_{k,p} = Y_p \cap E_k^{(p-1)} \ar[rr]& & E_k^{(p-1)}}
\end{equation*}

For any $k+1 \leq p \leq n$, let $m_p$ be the positive integer such that
$K_{X_p} = \pi_p^* K_{X_{p-1}} + m_p E_p$. By proposition \ref{canonical},
this is also the integer such that:

\begin{equation}
\label{equacanonical}
K_{E_k^{(p)}} = \pi_p^{*}|_{E_{k}^{(p)}} K_{E_k^{(p-1)}} + m_p E_{k,p}^{(p)}.
\end{equation}  

We define:

\begin{equation*}
B^k_{p,k_p} = {j_{k,p}^{(n)}}_* \left[q^*_{k,p} \DB(Y_{k,p}) \ot
\OO_{E_{k,p}^{(n)}} \left((m_p
-k_p) E_{k,p}^{(n)} + \sum_{i=p+1}^{n} m_i E_{k,i}^{(n)} \right) \right],
\end{equation*}
for all  $0 \leq k_p \leq m_p-1$.

Finally we let:

\begin{equation*}
C^k_{k,r_k} =  q_k^* \DB(Y_k) \ot \OO_{E_k^{(n)}} \left((m_k-r_k)E_k^{(n)} +
\sum_{i=k+1}^{n} m_i E_i^{(n)} \right),
\end{equation*}
for all $0 \leq r_k \leq m_k-1$.

\bigskip

The following result is the more precise version of proposition \ref{keyprop}
that we will prove:

\begin{prop} \label{keypropbis}

With the above hypothesis and notations, we
have a semi-orthogonal decomposition:

\begin{equation*}
\DB (X_n) = \langle A_{1,0}, \ldots, A_{1,m_1-1}, \ldots, A_{j,k_j}, \ldots,
A_{n,m_n-1},\, D_{X_n} \rangle,
\end{equation*}
where $D_{X_n}$ is the left orthogonal to the full admissible subcategory
generated by the $A_{j,k_j}$. Moreover, we have the inclusion $(\pi_1
\ldots \pi_n)^{*} \DP(X) \subset D_{X_n}$.

\bigskip

For each $1 \leq k \leq n$, we also have a semi-orthogonal decomposition:

\begin{equation*}
\DB (E_k^{(n)}) = \langle C^k_{k,0}, \ldots, C^k_{k,m_k-1}, B^k_{k+1,0}, \ldots,
B^k_{p,k_p}, \ldots, B^k_{n,m_n-1}, D_{E_k^{(n)}} \rangle, 
\end{equation*}
where $D_{E_k^{(n)}}$ is the left orthogonal to the subcategory generated by
the $C^k_{k, r_k }$ and the $B^k_{p, k_p}$. Moreover, we have the
inclusion $q_k^* \DB(Y_k) \subset D_{E_k^{(n)}}$. 

\end{prop}

\begin{proof} Assume that $n=1$. The map $ \pi_1 : X_1 \rightarrow X$ is the blow-up of
$X$ along $Y_1$. By definition of a wonderful resolution, the exceptional
divisor $E_1$ is
smooth, the variety $X_1$ is smooth, the map $\pi_1 : E_1 \rightarrow Y_1$ is
flat and $Y_1$ is smooth. Moreover, by definition of $m_1$ and by the adjunction
formula,
we have:
\begin{equation*}
K_{E_1} = q_1^*K_X|_{Y_1} + (m_1+1) E_1. 
\end{equation*}
So by lemma \ref{vanishing}, for all $k \in \mathbb{Z}$, we have:
\begin{equation*}
\RR {q_1}_* \RR \mathcal{H}om_{E_1}(\OO_{E_1}(kE_1), \OO_{E_1}(kE_1))= \OO_{E_1}.
\end{equation*}
Therefore, the categories
$\OO_{E_1}(kE_1) \ot q_1^* \DB (Y_1) = C_{1,m_1-k}^1$ are full admissible
subcategories of
$\DB(E_1)$, for all $k \in \mathbb{Z}$. Moreover, again by lemma
\ref{vanishing},
we have a semi-orthogonal decomposition:

\begin{equation*}
\DB (E_1) = \langle C_{1,0}^1, \ldots C_{1,m_1-1}^1, D_{E_1} \rangle, 
\end{equation*}
with the property $q_1^* \DB (Y_1) \subset D_{E_1}$.

\bigskip

Now, we can apply lemma \ref{divisor} to the embedding $i_1 : E_1
\hookrightarrow X_1$,
and we find a semi-orthogonal decomposition:

\begin{equation*}
\DB(X_1) = \langle A_{1,0},\ldots, A_{1,m_1-1}, D_{X_1} \rangle,
\end{equation*}
with $D_{X_1} = \{F \in \DB(X),\, i_1^*F \in D_{E_1} \}$. Let $G \in \DP(X)$. We
have $(\pi_1 i_1)^*G = q_1^*(G|_{Y_1})$. But $G|_{Y_1} \in \DB (Y_1)$, so
that
$q_1^*(G|_{Y_1}) \in D_{E_1}$, that is $i_1^*(\pi_1^*G) \in D_{E_1}$.
Thus $\pi_1^*(\DP (X)) \subset D_{X_1}$, which settles the proof of the
proposition in the case $n=1$.

\bigskip

Let $n \geq 2$ and assume that the proposition is true if $X$ admits a $n-1$
step resolution of singularities. We will prove that the proposition is true for a $n$-step wonderful resolution of singularities. Namely, let $X$ be an
algebraic variety with Gorenstein and canonical singularities and let:

\begin{equation*}
X_n \stackrel{\pi_n}\rightarrow X_{n-1} \ldots
X_1 \stackrel{\pi_1}\rightarrow X_0 = X,
\end{equation*}
be a $n$-step wonderful resolution of $X$. Let $1 \leq k \leq n$. By item $2$
and $4$ of definition \ref{wonderful}, we know that $E_k$ admits a
$(n-k)$-step wonderful resolution:

\begin{equation*}
E_k^{(n)} \stackrel{\Pi_{n,n-1}}\rightarrow \ldots
E_k^{(k+1)} \stackrel{\Pi_{k+1,k}}\rightarrow E_k^{(k)} = E_k,
\end{equation*}
which is a succession of blow-ups along the smooth normally flat centers
$Y_{k,k+1}, \ldots Y_{k,n}$. Thus by the recursion hypothesis and by formula
\ref{equacanonical}, we have a semi-orthogonal decomposition:
\begin{equation*}
\DB (E_k^{(n)}) = \langle B^k_{k+1,0}, \ldots, B^k_{k+1,m_{k+1}-1}, \ldots,
B^k_{p,k_p}, \ldots, B^k_{n,m_n-1}, \widetilde{D_{E_k^{(n)}}} \rangle,
\end{equation*}
with the inclusion $(\Pi_{k+1,k} \ldots \Pi_{n,n-1})^* \DP (E_k) \subset
\widetilde{D_{E_k^{(n)}}}$.

\bigskip
\textbf{\fbox{Step 1}}.
\smallskip

Now we want to prove the following:
\begin{claim} \label{claim1}
For $0 \leq r_k \leq m_k-1$, the categories:
\begin{equation*}
C^k_{k,r_k} =  q_k^* \DB(Y_k) \ot \OO_{E_k^{(n)}} \left((m_k-r_k)E_k^{(n)} +
\sum_{i=k+1}^{n} m_i E_i^{(n)} \right) 
\end{equation*}
are full admissible subcategories of $\DB
(E_k^{(n)})$, right orthogonal to one another and right
orthogonal to the subcategories $B^k_{p,k_p}$ for $k+1 \leq p \leq n$ and $0
\leq k_p \leq m_p-1$. Moreover, the category $q_k^* \DB (Y_k)$ is
left orthogonal to the subcategories $C^k_{r_k,k}$ and
$B^k_{p,k_p}$.
\end{claim}

\bigskip
The fact that the categories $C^k_{k,r_k}$ are full admissible subcategories of
$\DB (E_k^{(n)})$ which are right orthogonal to one another is a simple
variation of what we proved for the case $n=1$ of the proposition, the details
are left to the reader. In order to prove claim \ref{claim1}, we are
left to prove (with an obvious abuse of notation) that:

\begin{equation*}
 \begin{split}
  \HH(B^k_{p,k_p}, C^{k}_{k,r_k}) = 0 \\
  \HH(q_k^*\DB(Y_k), C^k_{k,r_k}) = 0 \\
  \HH(q_k^* \DB(Y_k), B^k_{p,k_p}) = 0
 \end{split}.
\end{equation*}
We start with the first vanishing, that is we want to prove that:

\begin{equation*}
\begin{split}
& \HH({j_{k,p}^{(n)}}_*(q_{k,p}^*\DB(Y_{k,p}) \ot \OO_{E_{k,p}^{(n)}}((m_p -k_p)
E_{k,p}^{(n)} + \sum_{i=p+1}^{n} m_i E_{k,i}^{(n)})), \\
& q_k^* \DB(Y_k) \ot \OO_{E_k^{(n)}}((m_k-r_k)E_k^{(n)} + \sum_{i=k+1}^{n}
m_i E_i^{(n)})) = 0.\\ 
\end{split}
\end{equation*}

By formula \ref{equacanonical} and the adjunction formula we have:

\begin{equation*}
K_{E_k^{(n)}} = (m_k+1) E_k^{(n)}|_{E_k^{(n)}} + \sum_{i=k+1}^{n} m_i
E^{(n)}_{k,i} + q_{k,n}^*K_{X_k}|_{Y_{k+1}}.
\end{equation*}

Thus, by Serre duality we have the equality (note that since all
intersections are proper, $\OO_{E_k^{(n)}}(E_i^{(n)}) =
\OO_{E_k^{(n)}}(E_{k,i}^{(n)})$):

\begin{equation*}
\begin{split}
& \HH({j_{k,p}^{(n)}}_*(q_{k,p}^*\DB(Y_{k,p}) \ot \OO_{E_{k,p}^{(n)}}((m_p -k_p)
E_{k,p}^{(n)} + \sum_{i=p+1}^{n} m_i E_{k,i}^{(n)})), \\
& q_k^* \DB(Y_k) \ot \OO_{E_k^{(n)}}((m_k-r_k)E_k^{(n)} + \sum_{i=k+1}^{n}
m_i E_i^{(n)}))  \\
& = \HH(q_k^* \DB(Y_k) \ot \OO_{E_k^{(n)}}((-1-r_k)E_k^{(n)}),\\
& {j_{k,p}^{(n)}}_*(q_{k,p}^*\DB(Y_{k,p}) \ot \OO_{E_{k,p}^{(n)}}((m_p -k_p)
E_{k,p}^{(n)}
+ \sum_{i=p+1}^{n} m_i E_{k,i}^{(n)})))^{\vee}.\\
\end{split}
\end{equation*}

Note that we used that
\begin{equation*}
q_k^* \DB(Y_k) \ot \OO_{E_k^{(n)}}\left((m_k-r_k)E_k^{(n)} + \sum_{i=k+1}^{n}
m_i E_i^{(n)} \right) =  q_k^* \DB(Y_k) \ot
\OO_{E_k^{(n)}}((-1-r_k)E_k^{(n)}) \ot K_{E_k^{(n)}}
\end{equation*}
in the above equality. By adjunction we also have:

\begin{equation*}
\begin{split}
 & \HH(q_k^* \DB(Y_k) \ot \OO_{E_k^{(n)}}((-1-r_k)E_k^{(n)}), \\
& {j_{k,p}^{(n)}}_*(q_{k,p}^*\DB(Y_{k,p}) \ot \OO_{E_{k,p}^{(n)}}((m_p -k_p)
E_{k,p}^{(n)}
+ \sum_{i=p+1}^{n} m_i E_{k,i}^{(n)})))  \\
 & = \HH((j_{k,p}^{(n)})^*(q_k^* \DB(Y_k) \ot
\OO_{E_k^{(n)}}((-1-r_k)E_k^{(n)})),
\\
 & q_{k,p}^*\DB(Y_{k,p}) \ot \OO_{E_{k,p}^{(n)}}((m_p -k_p) E_{k,p}^{(n)} +
\sum_{i=p+1}^{n} m_i E_{k,i}^{(n)})). \\
\end{split}
\end{equation*}

But $(j_{k,p}^{(n)})^* (q_k^* \DB(Y_k)) \subset q_{k,p}^* \DB(Y_{k,p})$ and
$j_p^*
\OO_{E_k^{(n)}}(E_k^{(n)}) = q_{k,p}^* \OO_{Y_{k,p}}(E_k^{(p-1)})$ because we
have a fibered diagram:

\begin{equation*}
 \xymatrix{
E_{k,p}^{(n)} \ar[rr]^{j_{k,p}^{(n)}} \ar[dd]_{q_{k,p}} & &
E_{k}^{(n)} \ar[dd]^{\pi_{p} \ldots \pi_n|_{E_k^{(n)}}} \\
& & \\
Y_{k,p}  \ar[rr]& & E_k^{(p-1)}}
\end{equation*}

As a consequence, to prove that:

\begin{equation*}
\begin{split}
& \HH((j_{k,p}^{(n)})^*(q_k^* \DB(Y_k) \ot
\OO_{E_k^{(n)}}((-1-r_k)E_k^{(n)})),\\
& q_{k,p}^*\DB(Y_{k,p}) \ot \OO_{E_{k,p}^{(n)}}((m_p -k_p) E_{k,p}^{(n)} +
\sum_{i=p+1}^{n} m_i E_{k,i}^{(n)}))= 0,\\
\end{split}
\end{equation*}
we only have to prove that:

\begin{equation}
\label{vanishing2}
{q_{k,p}}_* \left(\OO_{E_{k,p}^{(n)}}((m_p -k_p) E_{k,p}^{(n)} +
\sum_{i=p+1}^{n}
m_i E_{k,i}^{(n)}) \right) = 0.
\end{equation}

The map $q_{k,p} : E^{(n)}_{k,p} \rightarrow Y_{k,p}$ factors as
$\theta_{k,p}^n : E_{k,p}^{(n)} \rightarrow E_{k,p}^{(p)}$ followed by the
projection
$\widetilde{\Pi_{p,p-1}} : E_{k,p}^{(p)} \rightarrow Y_{k,p}$. But the map
$\theta_{k,p}^n$  is a succession of blow-ups
along the smooth normally flat centers $Y_{i} \cap E_{k,p}^{(i-1)}$, which
exceptional divisors on
$E_{k,p}^{(n)}$ are the $E_{k,i}^{(n)}|_{E_{k,p}^{(n)}}$, for $i=p+1 \ldots n$.
Moreover, by item $3$ and $4$ of definition \ref{wonderful} and
by proposition \ref{canonical}, we know that:

\begin{equation*}
K_{E_{k,p}^{(n)}} = (\theta_{k,p}^n)^* K_{E_{k,p}^{(p)}} + \sum_{i=p+1}^{n}
m_iE_{k,i}^{(n)}|_{E_{k,p}^{(n)}}. 
\end{equation*}
We apply corollary \ref{corvanishing} to the morphism $\theta_{k,p}^n :
E_{k,p}^{(n)} \rightarrow E_{k,p}^{(p)}$ to get:

\begin{equation*}
{\theta_{k,p}^n}_* \left(\OO_{E_{k,p}^{(n)}}((m_p -k_p) E_{k,p}^{(n)} +
\sum_{i=p+1}^{n} m_i E_{k,i}^{(n)}) \right) = \OO_{E_{k,p}^{(p)}}((m_p -k_p)
E_{k,p}^{(p)}),
\end{equation*}

so that

\begin{equation*}
{q_{k,p}}_* \left(\OO_{E_{k,p}^{(n)}}((m_p -k_p) E_{k,p}^{(n)} +
\sum_{i=p+1}^{n}
m_i E_{k,i}^{(n)}) \right) = \widetilde{\Pi_{p,p-1}}_* \OO_{E_{k,p}^{(p)}}((m_p
-
k_p)E_{k,p}^{(p)}).
\end{equation*}

Now, by formula \ref{equacanonical} and the adjunction formula:

\begin{equation*}
K_{E_{k,p}^{(p)}} = \widetilde{\Pi_{p,p-1}}^*{K_{E_k^{(p-1)}}}|_{Y_{k,p}} +
(m_p+1)E_{k,p}^{(p)}.
\end{equation*}

So by lemma \ref{vanishing}, we have the vanishing:

\begin{equation*} 
\widetilde{\Pi_{p,p-1}}_* \OO_{E_{k,p}^{(p)}}((m_p - k_p)E_{k,p}^{(p)}) = 0,
\end{equation*}
for all $0 \leq k_p \leq m_p-1$, which is precisely what we wanted.

\bigskip

To conclude step 1, we must show that:

\begin{equation*}
q_k^* \DB (Y_k) \subset \leftexp{\perp}{C_{k,r_k}} \,\, \text{and} \,\, q_k^*
\DB (Y_k) \subset \leftexp{\perp}{B^k_{p,k_p}}
\end{equation*}
for all $k+1 \leq p \leq n$. We have:

\begin{equation*}
\begin{split}
& \HH \left[q_k^*\DB(Y_k),(j_{k,p}^{(n)})_*\left(q^*_{k,p} \DB(Y_{k,p}) \ot
\OO_{E_{k,p}^{(n)}}((m_p -k_p) E_{k,p}^{(n)} + \sum_{i=p+1}^{n} m_i
E_{k,i}^{(n)}) \right) \right]  \\
& = \HH \left[(j_{k,p}^{(n)})^* q_k^*\DB(Y_k), q^*_{k,p} \DB(Y_{k,p}) \ot
\OO_{E_{k,p}^{(n)}} \left((m_p -k_p) E_{k,p}^{(n)} + \sum_{i=p+1}^{n} m_i
E_{k,i}^{(n)} \right) \right], \\
& \subset \HH \left[q_{k,p}^* \DB(Y_{k,p}), q^*_{k,p} \DB(Y_{k,p}) \ot
\OO_{E_{k,p}^{(n)}} \left((m_p -k_p) E_{k,p}^{(n)} + \sum_{i=p+1}^{n} m_i
E_{k,i}^{(n)} \right) \right],\\
& \text{since}\, (j_{k,p}^{(n)})^* q_k^* \DB(Y_k) \subset q_{k,p}^*
\DB(Y_{k,p}). \\
\end{split}
\end{equation*}

But we have already proved in formula \ref{vanishing2} that:
\begin{equation*}
{q_{k,p}}_*\OO_{E_{k,p}^{(n)}} \left((m_p -k_p) E_{k,p}^{(n)} + \sum_{i=p+1}^{n}
m_i E_{k,i}^{(n)} \right) = 0,
\end{equation*}

so that:

\begin{equation*}
\HH\left[q_{k,p}^* \DB(Y_{k,p}), q^*_{k,p} \DB(Y_{k,p}) \ot
\OO_{E_{k,p}^{(n)}} \left((m_p
-k_p) E_{k,p}^{(n)} + \sum_{i=p+1}^{n} m_i E_{k,i}^{(n)} \right) \right] =
0.
\end{equation*}
This proves that $q_k^* \DB (Y_k) \subset \leftexp{\perp}{B^k_{p,k_p}}$, for
all $k+1 \leq p \leq n$. The fact that $q_k^* \DB (Y_k) \subset
\leftexp{\perp}{C_{k,r_k}}$ is done in the same fashion. Step 1 is thus
complete, that is we have a semi-orthogonal decomposition:

\begin{equation*}
\DB (E_k^{(n)}) = \langle C^k_{k,0}, \ldots, C^k_{k,m_k-1}, B^k_{k+1,0}, \ldots,
B^k_{p,k_p}, \ldots, B^k_{n,m_n-1}, D_{E_k^{(n)}} \rangle, 
\end{equation*}
with the property : $q_k^* \DB(Y_k) \subset D_{E_k^{(n)}}$.

\bigskip

\textbf{\fbox{Step 2}}

\smallskip

From the above semi-orthogonal decomposition of $\DB(E_k^{(n)})$ for all $1 \leq
k \leq n$, we want to deduce a semi-orthogonal decomposition:

\begin{equation*}
\DB (X_n) = \langle A_{1,0}, \ldots, A_{1,m_1-1}, \ldots, A_{j,k_j}, \ldots,
A_{n,m_n-1},\, D_{X_n} \rangle,
\end{equation*}
with the property : $(\pi_1 \ldots \pi_n)^{*}
\DP(X) \subset D_{X_n}$, where $D_{X_n}$ is the left orthogonal to the
subcategory generated by the $A_{j,k_j}$.

\bigskip

Let $j \in [1,\ldots,n]$.  We note that $A_{j,k_j} = (i_j^{(n)})_*C^j_{j,k_j}$.
Lemma \ref{divisor} applied to the semi-orthogonal decomposition we found for
$\DB(E_j^{(n)})$  proves that for all
$0 \leq k_j \leq m_j-1$, the subcategories $A_{j,k_j}$ are admissible full
subcategories of $\DB(X_n)$ which are left orthogonal to one another. 
\bigskip

So, we are left to prove that for $1 \leq j < p \leq n$, for all $0 \leq k_j
\leq m_j-1$ and for all $0 \leq k_p \leq m_p-1$, we have:

\begin{equation*}
\begin{split}
& \HH({i_p^{(n)}}_*(q_p^{*} \DB (Y_p) \ot \OO_{E_p^{(n)}}((m_p-k_p) E_p^{(n)}
+ \sum_{t=p+1}^{n} m_t E_t^{(n)})),\\
& {i_j^{(n)}}_*(q_j^{*} \DB (Y_j) \ot \OO_{E_j^{(n)}}((m_j-k_j) E_j^{(n)}
+ \sum_{t=j+1}^{n} m_t E_t^{(n)}))) =0,\\
\end{split}
\end{equation*}

that is:

\begin{equation*}
\begin{split}
& \HH({i_j^{(n)}}^*{i_p^{(n)}}_*(q_p^{*} \DB (Y_p) \ot \OO_{E_p^{(n)}}((m_p-k_p)
E_p^{(n)}
+ \sum_{t=p+1}^{n} m_t E_t^{(n)})),\\
& q_j^{*} \DB (Y_j) \ot \OO_{E_j^{(n)}}((m_j-k_j) E_j^{(n)}
+ \sum_{t=j+1}^{n} m_t E_t^{(n)})) =0.\\
\end{split}
\end{equation*}
or with our notations:

\begin{equation*}
\begin{split}
& \HH\left[{i_j^{(n)}}^*{i_p^{(n)}}_* \left(q_p^{*} \DB (Y_p) \ot
\OO_{E_p^{(n)}}((m_p-k_p)
E_p^{(n)}
+ \sum_{t=p+1}^{n} m_t E_t^{(n)}) \right),C^j_{j,k_j} \right] =0.\\
\end{split}
\end{equation*}

But we have a fibered diagram:

\begin{equation*}
 \xymatrix{
 E_{j,p}^{(n)} \ar[rr]^{j_{j,p}^{(n)}} \ar[dd]_{j_{p,j}^{(n)}} & & E_j^{(n)}
\ar[dd]^{i_j^{(n)}} \\
& & \\
E_p^{(n)} \ar[rr]^{i_p^{(n)}}& & X_n}
\end{equation*}

The intersection $E_j^{(n)} \cap E_p^{(n)} = E_{j,p}^{(n)}$ is proper and all
varieties appearing in this fibered diagram are smooth so that this diagram is
exact cartesian, that is:

\begin{equation*}
{i_j^{(n)}}^*{i_p^{(n)}}_*(F) = {j_{j,p}^{(n)}}_*{j_{p,j}^{(n)}}^*(F) ,
\end{equation*}
for all $F \in \DB(E_p^{(n)})$. In particular we have:

\begin{equation*}
\begin{split}
& {i_j^{(n)}}^*{i_p^{(n)}}_*\left(q_p^{*} \DB (Y_p) \ot
\OO_{E_p^{(n)}}((m_p-k_p)
E_p^{(n)}
+ \sum_{t=p+1}^{n} m_t E_t^{(n)})\right) \\
& = {j_{j,p}^{(n)}}_* {j_{p,j}^{(n)}}^*\left(q_p^{*} \DB (Y_p) \ot
\OO_{E_p^{(n)}}((m_p-k_p) E_p^{(n)}
+ \sum_{t=p+1}^{n} m_t E_t^{(n)})\right)  \\
& = {j_{j,p}^{(n)}}_*\left(q_{j,p}^{*} \DB (Y_{j,p}) \ot
\OO_{E_{j,p}^{(n)}}((m_p-k_p)
E_{j,p}^{(n)}
+ \sum_{t=p+1}^{n} m_t E_{j,t}^{(n)})\right) \\
& = B^j_{p,k_p}
\end{split}
\end{equation*}

The semi-orthogonal decomposition we found for $\DB(E_j^{(n)})$ in step
1 shows that for $1 \leq j < p \leq n$, for all $0 \leq k_j
\leq m_j-1$ and for all $0 \leq k_p \leq m_p-1$:

\begin{equation*}
\begin{split}
& \HH(B^j_{p,k_p}, C^j_{j,k_j}) =0,\\
\end{split}
\end{equation*}
which is the vanishing we wanted. As a consequence, we have a semi-orthogonal
decomposition:

\begin{equation*}
\DB (X_n) = \langle A_{1,0}, \ldots, A_{1,m_1-1}, \ldots, A_{j,k_j}, \ldots,
A_{n,m_n-1},\, D_{X_n} \rangle.
\end{equation*}

The fact that $(\pi_1 \ldots \pi_n)^{*} \DP(X) \subset D_{X_n}$ is proved
easily, if one notices that for all $1 \leq j \leq n$:

\begin{equation*}
{i_j^{(n)}}^* \left((\pi_1 \ldots \pi_n)^{*} \DP(X)\right) \subset q_j^*
\DB(Y_k),
\end{equation*}
 and that for all $0 \leq k_j \leq m_j-1$:

\begin{equation*}
\HH\left[q_j^* \DB(Y_j), q_j^*\DB(Y_j) \ot
\OO_{E_j^{(n)}}\left((m_j-k_j)E_j^{(n)} +
\sum_{t=j+1}^{n} m_t E_t^{(n)}\right)\right] = 0,
\end{equation*}
since $q_j^* \DB(Y_j) \subset D_{E_j^{(n)}}$.
This concludes Step 2 and the recursive proof of the proposition.
\end{proof}

\end{subsection}

\end{section}

\newpage

\begin{section}{Conclusion : Minimality and further existence results}

\begin{subsection}{Minimality for categorical resolutions}

In this section we will discuss minimality of categorical crepant resolutions of singularities in some special settings. In \cite{kuz}, Kuznetsov conjectures the following:

\begin{conj} \label{conjkuz}
Let $X$ be an algebraic variety with rational Gorenstein singularities. Let $\T \rightarrow \DB(X)$ be a categorical strongly crepant resolution of $X$. Then, for any other categorical resolution $\T' \rightarrow \DB(X)$, there exists a fully faithful functor:

\begin{equation*}
\T \hookrightarrow \T'
\end{equation*} 
\end{conj} 
This is indeed a generalization of the Bondal-Orlov conjecture we mentioned in the introduction. Hence, it seems very interesting to look for categorical strongly crepant resolutions of singularities.

If $X$ admits a $1$-step wonderful resolution of singularities $\pi : \tilde{X}
\rightarrow X$, Kuznetsov relates the existence of a strongly crepant
categorical resolution of singularities to the existence of a rectangular
Lefschetz decomposition on the exceptional divisor of $\pi$ (see Theorem \ref{kuzmaintheo}).  It would be interesting to see if his techniques can be pushed in our context and to find more examples of varieties admitting strongly crepant categorical resolution of singularities.

\begin{quest}
 
What are the determinantal
varieties which can be proved to have strongly crepant categorical resolutions
of singularities?
 
\end{quest}

Kuznetsov answers positively to the question for the Pfaffian $\mathrm{Pf}_4 =
\mathbb{P}(\{ \omega \in \bigwedge^{2}V,\,
\operatorname{rank}(\omega) \leq 4 \}$, when $\dim V$ is odd
(see \cite{kuz}, section $8$).

\bigskip

Conjecture \ref{conjkuz} shows that strongly crepant resolution of singularities are expected to enjoy very strong minimality properties. But this conjecture seems to be highly non trivial and very difficult to check, even with the most basic examples \footnote{The hard point being that, in the setting of conjecture \ref{conjkuz}, we have absolutely no clue how to construct the functor $\T \hookrightarrow \T'$.}. Nevertheless, there is a slightly different, certainly easier to check, point of view on minimality for a resolution of singularities. 

\begin{defi} \label{weakminidef}
 Let $X$ be an algebraic variety with Gorenstein and rational singularities. let
$ {\pi}_1 : \tilde{X}_1 \rightarrow X$ be a resolution of singularities of $X$
and let $\d_1 : \T_1
\hookrightarrow \DB (\tilde{X}_1)$ be a categorical
resolutions of singularities of $X$. We say that $\T_1$ is \textbf{weakly
minimal} if for any other resolution of singularities $ {\pi}_2 : \tilde{X}_2
\rightarrow X$ with a morphism $\pi_{12} : \tilde{X}_1 \rightarrow \tilde{X}_2$
and a commutative diagram:

\begin{equation*}
\xymatrix{
\DB(\tilde{X}_1) \ar[dd]^{{\pi_1}_*} \ar[rd]^{{\pi_{12}}_*} & \\
& \DB(\tilde{X}_2) \ar[ld]^{{\pi_2}_*} \\
\DB(X) & \\
}
\end{equation*}

and any other categorical resolution $\d_2 : \T_2
\hookrightarrow \DB (\tilde{X}_2)$ with a commutative diagram:

\begin{equation*}
\xymatrix{
\T_1 \ar[dd]^{{\pi_{\T_1}}_*} \ar[rd]^{{{\pi_{{\T_1},{\T_2}}}_*}} & \\
& \T_2 \ar[ld]^{{\pi_{\T_2}}_*} \\
\DB(X) & \\
}
\end{equation*}

with ${\pi_{\T_i}}_* = {\pi_i}_*\d_i$, ${\pi_{{\T_1},{\T_2}}}_* = \d_2^!
{\pi_{12}}_* \d_1$ and such that $\T_1$ and $\T_2$ are $\DP (X)$-module
categories, we have the implication:
\begin{equation*}
\p12^* \text{is fully faithful} \implies \T_1 = \T_2.
\end{equation*}

\end{defi}

Recall that given any triangulated category $\A$, we say that it is
\textit{connected} if it does not split as a sum $\A = \A_1 \oplus \A_2$, where
the $\A_i$ are both non trivial and totally orthogonal to each other.

\begin{prop}
 \label{weakmini}
Let $X$ be a projective algebraic variety with Gorenstein and rational singularities. let
$ {\pi}_1 : \tilde{X}_1 \rightarrow X$ be a resolution of $X$ and let $\d_1 :
\T_1
\hookrightarrow \DB (\tilde{X}_1)$ be a categorical
resolution of singularities of $X$. If $\T_1$ is strongly crepant and
connected, then
$\T_1$ is weakly minimal. 
\end{prop}

This is a categorical generalization of the following result (which follows
from the ramification formulas). Let $X$ be an algebraic variety with Gorenstein
rational singularities. Let $\pi_1 : \tilde{X} \rightarrow X$ be a crepant
resolution of singularities of
$X$ and let $ \pi_2 : Y \rightarrow X$ be any other resolution of singularities
of $X$. Assume that there is a morphism $ \pi_{12} : \tilde{X} \rightarrow Y$
which makes the following triangle commutative:

\begin{equation*}
\xymatrix{
\tilde{X} \ar[dd]^{\pi_1} \ar[rd]^{\pi_{12}} & \\
& Y \ar[ld]^{\pi_2} \\
X & \\
}
\end{equation*}
then $\tilde{X} = Y$. 

Note that if $\T \rightarrow \DB(X)$ is a categorical strongly crepant resolution of $X$, then there exists a full admissible connected subcategory $\T' \subset \T$, such that $\T' \rightarrow \DB(X)$ is a strongly crepant resolution (see Lemme $1.3.9$ of \cite{theseabuaf}). Finally, we do expect that the projectivity assumption should be removed, but the proof in the non-projective case seems to be quite annoying. One would need to introduce semi-orthogonal decompositions of a derived category with respect to a base scheme and prove all the basic results concerning semi-orthogonal decomposition in that setting. We leave it as an open question for the reader.

\begin{proof}[of proposition \ref{weakmini}]
Let $\pi_2 : \tilde{X_2} \rightarrow X$ and $\d_2 \hookrightarrow \DB
(\tilde{X}_2)$ be another categorical resolution of singularities of $X$ for
which we have the commutative diagrams of definition \ref{weakminidef}.
Since ${\p12}^*$ is fully faithful and has a left adjoint, we have a
semi-orthogonal decomposition:

\begin{equation*}
 \T_1 = \langle ({\p12}^*\T_2)^{\perp}, {\p12}^*\T_2 \rangle.
\end{equation*}
Since $\T_1$ is connected, to prove that $\T_1 = \T_2$, we only have to prove
that the above decomposition is totally orthogonal. This means that we have to
prove that for all $a \in (\p12^*\T_2)^{\perp}$ and for all $b
\in \p12^*\T_2$:

\begin{equation*}
\HH(a,\p12^*b)=0. 
\end{equation*}
 But 
\begin{equation*}
\begin{split}
& \HH(a, \p12^*b) \\
& = \HH(\p12^*b, S_{\T_1}(a))^{\vee}, \text{by Serre
duality,}\\
& =\HH(\d_1^* \pi_{12}^* \d_2 b, S_{\T_1}(a))^{\vee}, \text{by definition of}\,
{\p12}_*,\\
& =\HH (\pi_{12}^* \d_2 b, \d_1 S_{\T_1}(a))^{\vee}, \text{by adjunction,} \\
& = \HH(\pi_{12}^* \d_2 b, \pi_1^*K_X[\dim X] \ot \d_1 a)^{\vee},
\text{because}\, \T_1 \, \text{is a strongly crepant resolution of}\, X,\\
& = \HH(\pi_{12}^* \d_2 b \ot \pi_{12}^* \pi_2^*(K_X^{-1})[-\dim X], \d_1
a)^{\vee},\\
& = \HH (\pi_{12}^* \d_2 b', \d_1 a)^{\vee}, \text{with}\, b'\in \T_2,\,
\text{because}\, \T_2\, \text{is a}\, \DP (X)\text{-module category,} \\
& = \HH (\d_1^* \pi_{12}^* \d_2 b', a)^{\vee}, \text{by adjunction,}\\
& = \HH(\p12^* b', a)^{\vee} \\
& =0, \text{as}\, a \in (\p12^*\T_2)^{\perp}.\\
\end{split}
\end{equation*}

\end{proof}

\end{subsection}

\begin{subsection}{Existence results for prehomogeneous spaces}

In section $2.1$ of the present paper, we describe some varieties with Gorenstein rational singularities which have a wonderful resolution of singularities. As a consequence of our main theorem, they admit a categorical weakly crepant resolution of singularities. These examples fit very well into the theory of
\textit{reductive prehomogeneous vector spaces}. We recall that a reductive
prehomogeneous vector space is the data $(\mathrm{G},V)$ of a reductive linear
group $\mathrm{G}$
and a finite dimensional vector space $V$ such that $\mathrm{G}$ acts on $V$
with a dense orbit. For instance, the determinantal varieties defined by the
minors of the generic square (resp. symmetric, resp. skew-symmetric) $n
\times n$ matrix are the orbit closures of the action of $\mathrm{GL_n \times
GL_n}$ (resp. $\mathrm{GL_n}$, resp. $\mathrm{GL_n}$) on $V \otimes V$ (resp.
$S^2(V)$, resp. $\bigwedge^{2}V$). As for the affine cone over $\mathbb{OP}^2 =
\mathrm{E_6/P_1}$ and
its secant variety, they are the orbit closures of the action of $\mathbb{C}^* \times \mathrm{E_6}$ on $V_{\omega_1}$, where $\omega_1$ is the weight associated to $\mathrm{P_1}$. So one is tempted to make the following conjecture:

\begin{conj}
Let $(\mathrm{G},V)$ be a prehomogeneous vector space. Let $Z \subset V$ be the
closure of an orbit of $\mathrm{G}$. Assume that $Z$ has Gorenstein rational singularities. Then $Z$ admits a categorical weakly crepant resolution of singularities.
\end{conj}
To investigate this conjecture in more details, one can ask the following:

\begin{quest}
Let $(\mathrm{G},V)$ be a prehomogeneous vector space. Let $Z \subset V$ be the
closure of an orbit of $\mathrm{G}$. When does $Z$ admit a
wonderful resolution of singularities?
\end{quest}

The following example shows that the answer to the above question can not be
"always".

\begin{exem}[Tangent variety of $\mathrm{Gr(3,6)}$] \label{G(3,6)}
 
\upshape{Let $V$ be a vector space with $\dim V = 6$ and let $W =
\mathrm{Gr(3,V)} \subset \mathbb{P}(\bigwedge^3 V)$ be
the Grassmannian of $\mathbb{C}^3 \subset V$ inside its Pl\"ucker
embedding. We can decompose $\bigwedge^3 V$ as:

\begin{equation*}
\mathbb{C} \oplus U \oplus U^* \oplus \mathbb{C}, 
\end{equation*}
where $U$  is identified with the space of $3 \times 3$ matrices, (see
\cite{manilands} section $5$ for more details). We denote by $C$ the
determinant on $U$, which can be seen as a map $S^3 U \rightarrow \mathbb{C}$
or as a map $S^2 U \rightarrow U^*$. We also denote by $C^*$ the
determinant on $U^*$.

Let $Z$ denote the tangent variety to $W$. It is shown in \cite{manilands} that
an equation (up to an automorphism of
$\mathbb{P}(\bigwedge^3 V)$) of $Z$ is:
\begin{equation*}
Q(x,X,Y,y) = (3xy - \frac{1}{2} \langle X,Y \rangle)^2 +
\frac{1}{3}(yC(X^{\otimes 3}) +
xC^*(Y^{\otimes 3}) - \frac{1}{6}\langle C^*(Y^{\otimes 2}),C(X^{\otimes 2})
\rangle
\end{equation*}
where $\langle .,. \rangle$ is the standard pairing between $U$ and
$U^*$. The partial derivatives of $Q$ give
the equations of the variety of ``stationary secants'' to $W$, which we denote
by
$\sigma_+(W)$. A simple computation of the Taylor expansion of $Q$ shows
that the variety $\sigma_+(W)$ is singular precisely along $W$, but
contrary to what is claimed in proposition $5.10$ of \cite{manilands}, $W$ is
not defined by all the second derivatives of $Q$. The orbit closures structure
of the
action of $\mathrm{SL_6}$ on $\mathbb{P}(\bigwedge^3 V)$ is the following:

\begin{equation*}
W \subset \sigma_+(W) \subset Z \subset \mathbb{P}(\bigwedge^3 V).
\end{equation*}

As a consequence, the only "natural" \footnote{One would of course like this resolution to be also $\mathrm{SL_6}$-equivariant. So we must start with the blow-up of a $\mathrm{SL_6}$-invariant subvariety.} procedure to get a wonderful resolution of singularities of
$Z$ would be to consider the blow-up of $Z$ along $W$:
\begin{equation*}
 \pi_1 : Z_1 \rightarrow Z 
\end{equation*}
and then the blow-up of $Z_1$ along the strict transform of $\sigma_+(W)$:
\begin{equation*}
 \pi_2 : Z_2 \rightarrow Z_1. 
\end{equation*}
One can check that the
proper transform of $\sigma_+(W)$ under $\pi_1$, the exceptional divisor of
$\pi_2$ and the variety $Z_2$ are smooth. But an easy computation shows that the
tangent cone of $Z$
along any point of $W$ is a double hyperplane, so that the exceptional
divisor $E_1$ of the blow-up of $Z$ along $W$ is globally non reduced. The
strict transform of $E_1$ after any normally flat blow-up of
$Z_1$ will still be globally non reduced. As a
consequence the above sequence of blow-ups will never produce a wonderful
resolution of singularities of $Z$.

\smallskip

Now, consider $p = (p_0,P_0,P_1,p_1)$ a generic point of
$\mathbb{P}(\bigwedge^3 V)$
and let $\P(Q,p)$ be the polar equation of $Q$ with respect to $p$,
that is:

\begin{equation*}
 \P(Q,p) = p_0\frac{\partial Q}{\partial x} + P_0\frac{\partial Q}{\partial X} +
P_1\frac{\partial Q}{\partial Y} + p_1\frac{\partial Q}{\partial y}
\end{equation*}

It is easily noticed that the cubic hypersurface (which we also denote by
$\P(Q,p)$) defined by this equation is smooth and it
contains
$\sigma_+(W) = Z_{sing}$. For any $w \in \sigma_+(W)-W$, the tangent space of
$\P(Q,p)$ at $w$ is transverse to the tangent cone to $Z$ at $w$, so that the
tangent cone to $\P(Z,p) = \P(Q,p) \cap Z$ at $w$ is a cone over a smooth
quadric of dimension $4$ with vertex the embedded tangent space to
$\sigma_+(W)$ at $w$. For any $w \in W = \mathrm{Gr(3,6)}$, the
the tangent space of $\P(Q,p)$ at
$w$ is equal to the reduced tangent cone to $Z$ at $w$. Thus, looking at the
Taylor expansion of $Q$ at $w$, one can prove that
the tangent cone to $\P(Z,p)$ at $w$ is the secant variety of a cone over
$\mathbb{P}^2 \times \mathbb{P}^2$ (this cone over $\mathbb{P}^2 \times
\mathbb{P}^2$
being the set of $\mathbb{C}^3 \subset V$ which intersect $w$ in
dimension at least $2$).
The vertex of this cone is the embedded tangent space to $\mathrm{Gr(3,6)}$ at
$w$ and this cone is singular precisely along the cone over $\mathbb{P}^2
\times \mathbb{P}^2$ (see \cite{manilands} for instance).

\smallskip

Note that the tangent cone to $\P(Z,p)$ at any point $w \in \mathrm{Gr(3,6)}$
does not depend on the choice of a generic $p \in \mathbb{P}(\bigwedge^3 V)$, as
predicted by the theory of L\^e-Teissier (see \cite{abu1}, section $2.2$ for
some
recollections on the theory of L\^e-Teissier in the setting of projective
geometry or \cite{lete1} and \cite{lete2} for the theory in its general
setting). The above description of the tangent cones of $\P(Z,p)$ along its
various strata shows that if one considers the blow-up of $\P(Z,p)$ along
$W$:
\begin{equation*}
 \pi_1 : \P_1 \rightarrow \P(Z,p), 
\end{equation*}
and then the blow-up of $\P_1$ along the strict transform of
$\sigma_+(W)$ through $\pi_1$:
\begin{equation*}
 \pi_2 : \P_2 \rightarrow \P_1, 
\end{equation*}
then one gets a wonderful resolution of singularities of $\P(Z,p)$.}
\end{exem}

The above example suggests the following conjecture.

\begin{conj} Let $(\mathrm{G},V)$ be a prehomogeneous vector space. Let $Z
\subset V$ be the closure of an orbit of $\mathrm{G}$. There is an integer $d
\leq \dim V$ such that for a generic $L \in \mathrm{G(d,\dim V)}$, the polar
$\P(Z,L)$ contains $Z_{sing}$ and admits a wonderful resolution of
singularities.
\end{conj}

Finally, let us mention that \cite{deliu} undergoes a thorough study of a
possible homological projective dual of $\mathrm{Gr(3,V)} \subset
\mathbb{P}(\bigwedge^3 V)$ for $\dim V = 6$. Such a homological dual is expected
to be a  categorical crepant resolution of singularities of the double cover
of $\mathbb{P}(\bigwedge^3 V^*)$ ramified along the projective dual of
$\mathrm{Gr(3,V)}$ (which is equal to the tangent variety of $\mathrm{Gr(3,V^*)}
\subset \mathbb{P}(\bigwedge^3 V^*)$). However, from \cite{deliu}, it is not clear that a categorical crepant resolution of this double cover does
exist. Nevertheless, we strongly believe that the existence of such a
categorical crepant resolution should be linked to the existence of a
categorical crepant resolution of the dual variety of $\mathrm{Gr(3,V)}$, which in turn should be linked to the wonderful resolution of its generic polar. We come back to this circle of questions in \cite{abuaf-titsfreudenthal}.

\end{subsection}
\end{section}

\newpage

\bibliographystyle{alpha}

\bibliography{biblicrepant}

\begin{thebibliography}{BLVdB11}

\bibitem[Abu11]{abu1}
Roland Abuaf.
\newblock {Singularities of the projective dual variety.}
\newblock {\em Pac. J. Math.}, 253(1):1--17, 2011.

\bibitem[Abu13a]{abuaf-titsfreudenthal}
Roland Abuaf.
\newblock {Categorical crepant resolutions of singularities and the
  Tits-Freudenthal magic square}.
\newblock {arXiv:1307.1675}, 2013.

\bibitem[Abu13b]{theseabuaf}
Roland Abuaf.
\newblock {Dualit\'e projective homologique et r\'esolutions cat\'egoriques des
  singularit\'es}.
\newblock Th\`ese pour le grade de Docteur es Sciences de l'Universit\'e de
  Grenoble, 2013.

\bibitem[BLVdB10]{blv}
Ragnar-Olaf Buchweitz, Graham~J. Leuschke, and Michel Van~den Bergh.
\newblock Non-commutative desingularization of determinantal varieties {I}.
\newblock {\em Invent. Math.}, 182(1):47--115, 2010.

\bibitem[BLVdB11]{blvII}
Ragnar-Olaf Buchweitz, Graham~J. Leuschke, and Michel Van~den Bergh.
\newblock {Non-commutative desingularization of determinantal varieties II.}
\newblock {arXiv:1106.1833v1}, 2011.

\bibitem[BO02]{BO}
A.~Bondal and D.~Orlov.
\newblock {Derived categories of coherent sheaves.}
\newblock {Li, Ta Tsien (ed.) et al., Proceedings of the international congress
  of mathematicians, ICM 2002, Beijing, China, August 20-28, 2002. Vol. II:
  Invited lectures. Beijing: Higher Education Press. 47-56 (2002).}, 2002.

\bibitem[CF07]{fu-chap}
Pierre-Emmanuel Chaput and Baohua Fu.
\newblock On stratified {M}ukai flops.
\newblock {\em Math. Res. Lett.}, 14(6):1055--1067, 2007.

\bibitem[DCP83]{procesi-decon}
C.~De~Concini and C.~Procesi.
\newblock Complete symmetric varieties.
\newblock In {\em Invariant theory ({M}ontecatini, 1982)}, volume 996 of {\em
  Lecture Notes in Math.}, pages 1--44. Springer, Berlin, 1983.

\bibitem[Del11]{deliu}
Dragos Deliu.
\newblock {Homological Projective Duality for Gr(3,6)}.
\newblock Dissertation for the degree of Doctor in Philosophy at the University
  of Pennsylvania. Available at
  http://repository.upenn.edu/dissertations/AAI3463052, 2011.

\bibitem[FM12]{manifaen}
Daniele Faenzi and Laurent Manivel.
\newblock {On the derived category of the Cayley plane II}.
\newblock {arXiv:1201.6327v1}, 2012.

\bibitem[Hir64]{hiro}
Heisuke Hironaka.
\newblock Resolution of singularities of an algebraic variety over a field of
  characteristic zero. {I}, {II}.
\newblock {\em Ann. of Math. (2) 79 (1964), 109--203; ibid. (2)}, 79:205--326,
  1964.

\bibitem[Hur11]{hurug}
Mathieu Huruguen.
\newblock {Compactification d'espaces homog\`enes sph\'eriques sur un corps
  quelconque}.
\newblock Th\`ese pour le grade de Docteur es Sciences de l'Universit\'e de
  Grenoble. Available at http://tel.archives-ouvertes.fr/tel-00716402, 2011.

\bibitem[Kle80]{kleiman}
Steven~L. Kleiman.
\newblock {Relative duality for quasi-coherent sheaves.}
\newblock {\em Compos. Math.}, 41:39--60, 1980.

\bibitem[KMM87]{KMM}
Yujiro Kawamata, Katsumi Matsuda, and Kenji Matsuki.
\newblock {Introduction to the minimal model problem.}
\newblock {Algebraic geometry, Proc. Symp., Sendai/Jap. 1985, Adv. Stud. Pure
  Math. 10, 283-360 (1987).}, 1987.

\bibitem[Kol97]{kollar}
J\'anos Koll\'ar.
\newblock {Singularities of pairs.}
\newblock {Koll\'ar, J\'anos (ed.) et al., Algebraic geometry. Proceedings of
  the Summer Research Institute, Santa Cruz, CA, USA, July 9--29, 1995.
  Providence, RI: American Mathematical Society. Proc. Symp. Pure Math.
  62(pt.1), 221-287 (1997).}, 1997.

\bibitem[Kuz06]{kuz4}
Alexander Kuznetsov.
\newblock {Hyperplane sections and derived categories.}
\newblock {\em Izv. Math.}, 70(3):447--547, 2006.

\bibitem[Kuz07]{kuz2}
Alexander Kuznetsov.
\newblock Homological projective duality.
\newblock {\em Publ. Math. Inst. Hautes \'Etudes Sci.}, (105):157--220, 2007.

\bibitem[Kuz08]{kuz}
Alexander Kuznetsov.
\newblock Lefschetz decompositions and categorical resolutions of
  singularities.
\newblock {\em Selecta Math. (N.S.)}, 13(4):661--696, 2008.

\bibitem[LM01]{manilands}
J.~M. Landsberg and L.~Manivel.
\newblock The projective geometry of {F}reudenthal's magic square.
\newblock {\em J. Algebra}, 239(2):477--512, 2001.

\bibitem[LT88]{lete1}
D{\~u}ng~Tr{\'a}ng L{\^e} and Bernard Teissier.
\newblock Limites d'espaces tangents en g\'eom\'etrie analytique.
\newblock {\em Comment. Math. Helv.}, 63(4):540--578, 1988.

\bibitem[Sch77]{schaub}
D.~Schaub.
\newblock Propri\'et\'e topologique du gradu\'e associ\'e d'un anneau
  {$(S_{r})$}.
\newblock {\em Comm. Algebra}, 5(11):1223--1239, 1977.

\bibitem[Tei82]{lete2}
Bernard Teissier.
\newblock Vari\'et\'es polaires. {II}. {M}ultiplicit\'es polaires, sections
  planes, et conditions de {W}hitney.
\newblock In {\em Algebraic geometry ({L}a {R}\'abida, 1981)}, volume 961 of
  {\em Lecture Notes in Math.}, pages 314--491. Springer, Berlin, 1982.

\bibitem[Tha99]{thad}
Michael Thaddeus.
\newblock Complete collineations revisited.
\newblock {\em Math. Ann.}, 315(3):469--495, 1999.

\bibitem[Vai84]{vain}
Israel Vainsencher.
\newblock Complete collineations and blowing up determinantal ideals.
\newblock {\em Math. Ann.}, 267(3):417--432, 1984.

\bibitem[Wey03]{weyman}
Jerzy~M. Weyman.
\newblock {\em {Cohomology of vector bundles and syzygies.}}
\newblock {Cambridge Tracts in Mathematics 149. Cambridge: Cambridge University
  Press. xiv, 371~p.}, 2003.

\bibitem[WZ12]{weyII}
Jerzy Weyman and Gufang Zhao.
\newblock {Noncommutative desingularization of orbit closures for some
  representations of $GL_n$}.
\newblock {arXiv:1204.0488v1}, 2012.

\bibitem[Zak93]{zak}
F.L. Zak.
\newblock {\em {Tangents and secants of algebraic varieties.}}
\newblock {Translations of Mathematical Monographs. 127. Providence, RI:
  American Mathematical Society (AMS). vii, 164 p.}, 1993.

\end{thebibliography}

\end{document}